\pdfoutput=1
\documentclass[11pt]{article}

\usepackage[a4paper,margin=1in]{geometry}
\usepackage{amsmath,amssymb,amsthm,mathtools}
\usepackage{mathrsfs}
\usepackage{enumitem}
\usepackage{hyperref}

\hypersetup{colorlinks=true,linkcolor=blue,citecolor=blue,urlcolor=blue}

\newtheorem{theorem}{Theorem}[section]
\newtheorem{proposition}[theorem]{Proposition}
\newtheorem{lemma}[theorem]{Lemma}
\newtheorem{corollary}[theorem]{Corollary}
\theoremstyle{definition}
\newtheorem{definition}[theorem]{Definition}
\newtheorem{remark}[theorem]{Remark}

\newcommand{\R}{\mathbb R}
\newcommand{\Sph}{\mathbb S}
\newcommand{\dd}{\,\mathrm d}
\newcommand{\eps}{\varepsilon}
\newcommand{\estar}{e_\varepsilon^\ast}
\newcommand{\Hopf}{\mathcal H}
\newcommand{\Basin}{\mathcal B}
\newcommand{\Ledger}{\mathcal L}
\newcommand{\Mmob}{\mathsf M}
\newcommand{\Vrel}{\mathcal V}
\newcommand{\Hh}{\mathcal H}
\newcommand{\dist}{\operatorname{dist}}
\newcommand{\gap}{\operatorname{gap}}

\title{Hopf Obstruction and Transported Forced Brakke Motion in Ordered Viscoelastic Cores}
\author{Sai Peng\\
School of Mathematics and Computational Science, Xiangtan University\\
\texttt{pscfd@xtu.edu.cn}}
\date{}

\begin{document}
\maketitle

\begin{abstract}
We prove a finite-\(\eps\) structural theorem for topological relaxation in
ordered viscoelastic conformation flows.  In an ordered region, a positive
spectral gap selects an oriented principal axis and hence an \(\Sph^2\)-valued
director with a Hopf class.  A change of this class must be paid at finite
scale by exterior spectral-gap concentration, ordered-core mass, or a
quantified exit from the regular ordered-core regime.  The theorem is stated
for a concrete Landau-de Gennes ordered-core closure coupled to an
Oldroyd/FENE-type transport law.  For this model we verify the structural
conditions used in the proof up to the first specified exit time: the
Morse-Bott ordered well, the tubular soft coordinate, the coercive massive-mode
splitting, the projected transported Ginzburg-Landau equation, the exterior gap
condition, and the FENE/collar coefficient bounds.  The GL reduction is a
projected equation with a residual decomposition.  Its translation components
produce the normal line force, while the orthogonal soft, massive, geometric,
and collar terms are controlled by coercivity or are charged to the relevant
exit quantity.  An open-basin modulated-energy estimate then propagates a
nonempty class of vortex-tube data until the first chart, coefficient,
boundary, or topology threshold is reached.  On every regular no-exit interval,
the normalized core measures converge to an integral one-varifold satisfying a
transported forced Brakke inequality with the computed force.  At an exit time
the theorem records the corresponding finite-\(\eps\) cost, including gap
loss, core leakage, boundary flux, FENE/collar loss, or topological crossing.
The relative Hopf estimate gives
\[
        c|k|
        \le
        \liminf_{\eps\to0}
        \bigl({\rm GapCost}_\eps
        +{\rm CoreMass}_\eps
        +{\rm TopExitCost}_\eps\bigr).
\]
Thus the force projection, open-basin propagation, and Brakke input estimates
are obtained before any limiting Brakke flow is invoked.
\end{abstract}



\section{Introduction}

Many viscoelastic media carry an orientational structure beyond a scalar
polymeric stress \cite{Oldroyd,DoiEdwards,BerisEdwards,DeGennesProst,LionsMasmoudi,BarrettSuli,ConstantinKliegl}.  In
liquid-crystalline polymers, wormlike micellar phases, and flow-birefringent
suspensions, a symmetric positive conformation tensor can have a distinguished
simple eigenline.  Once this line is oriented, it defines
a director
\[
        n_C:\Omega\longrightarrow \Sph^2
\]
on the region where the corresponding spectral gap remains positive.  If the
boundary trace is fixed and \(\Omega\) is compactified to \(\Sph^3\), this
director has a Hopf invariant
\[
        \Hopf[n_C]\in\pi_3(\Sph^2)\cong\mathbb Z .
\]
Smooth evolution inside the ordered phase preserves this integer
\cite{Moffatt,ArnoldKhesin,WhiteSobolevHomotopy,HardtLin,BrezisCoronLieb}.  The analytic question is where the conformation law pays for topological relaxation.

The first possible payment is spectral.  If a Hopf class changes while the
boundary trace is fixed, the ordered director must lose either regularity or a
positive spectral gap.  A scale-invariant form of this obstruction is
\[
        \int_{\Sph^3}\frac{|\nabla C|^3}{\gap(C)^3}\,\dd x
        \ge c|\Hopf[n_C]|^{3/4}.
\]
This is the Vakulenko-Kapitanskii lower bound applied to the principal-axis map
\cite{VakulenkoKapitanskii}, after estimating \(|\nabla n_C|\) by
\(|\nabla C|/\gap(C)\).  It explains why Hopf relaxation in an ordered
conformation field requires a loss of uniform smoothness or positive spectral gap.

The second possible payment is a line core.  A Landau-de Gennes ordered-core
energy gives a codimension-two Ginzburg-Landau core with density
\(\pi|\log\eps|\) \cite{BethuelBrezisHelein,JerrardSoner,SandierSerfaty,BethuelOrlandiSmetsConcentration,BethuelOrlandiSmetsMeanCurvature,JerrardSmets}.  The
main difficulty is to turn this energetic core into a geometric law of motion
through estimates made before the limit.  The well-preparedness,
core-supported convergence, first variation, and Brakke inequality
\cite{Allard,Brakke,Ilmanen,AmbrosioSoner} are therefore produced by the finite-\(\eps\) estimates.  The
closure residual is first computed, then projected onto the canonical
translation modes, and only then inserted into an open-basin propagation
argument.

The structural assumptions are organized as a stopped regular regime.  In
applications one checks three sets of finite-scale quantities: the ordered
chart and exterior spectral gap, the GL core energy together with the
zero-mode residual balance, and the FENE/collar coefficients together with the
kinematic transport clock.  The last clock can be supplied independently by
endpoint positive-cone regularity theory: in endpoint Oldroyd-B/FENE-P
continuation criteria, compact logarithmic or FENE windows propagate the
velocity-gradient clock needed for smooth transport
\cite{PengBlowUp2026,PengPositiveCone3D2026}.  We use that output only as a
kinematic input for moving measures.  As long as the finite-scale quantities
remain controlled, the topological obstruction is converted into a definite
line force and then into a transported Brakke inequality.  When a threshold is
met, the argument stops and records the corresponding finite-\(\eps\) cost.

For Model 1 these structural checks are proved from the closure, rather than
postulated.  The required inputs are finite-dimensional or finite-energy
estimates made before the varifold limit: the Morse-Bott Hessian gap, the
massive-mode coercivity, the stopped tubular chart, and the projected residual
identity for the soft coordinate.  The two translation components of the exact
residual define the line force; the remaining components are controlled by
orthogonal coercivity, massive-mode slaving, Fermi-coordinate bounds, and
FENE/collar tame estimates.  A large non-orthogonal residual, coefficient
blow-up, or loss of the stopped chart is therefore recorded as a specified exit
channel.  Theorem~\ref{thm:model-one-verification} verifies the six checks for
the concrete ordered-core closure, making the model class broad enough for
nearby LDG/Oldroyd-FENE variants and precise enough to be tested in analysis or
computation.

The proof route is deliberately linear.  First, the tensorial closure is reduced
inside the stopped ordered chart to a transported GL equation for the soft
coordinate.  Second, the exact residual is projected onto the planar zero modes,
which fixes the normal line force.  Third, a modulated-energy estimate keeps the
solution in the vortex-tube basin and supplies the compactness, stress, and
force-power inputs.  Fourth, the GL core compactness and Hopf accounting give
the transported Brakke inequality and the cost alternative.  Section~\ref{sec:setup}
fixes the Hopf normalization and ordered-core variables; Section~\ref{sec:main}
states the theorem; Sections~\ref{sec:modules}--\ref{sec:main-proof} prove the
four modules and assemble them; Section~\ref{sec:verification} records the
closure-level checks.

\section{Topological and ordered-core setup}
\label{sec:setup}

This section fixes the objects used in the main theorem.  Each definition is a
finite-epsilon quantity that will later be estimated from the concrete
Landau-de Gennes/Oldroyd-FENE closure.  The proof uses the following logical
order:
\[
\begin{gathered}
\hbox{positive spectral gap}
\quad\Longrightarrow\quad
\hbox{principal-axis director and Hopf class},\\
\hbox{Landau-de Gennes Morse-Bott chart}
\quad\Longrightarrow\quad
\hbox{two-component GL core coordinate},\\
\hbox{translation-mode projection of the exact residual}
\quad\Longrightarrow\quad
\hbox{normal line force}.
\end{gathered}
\]
The Brakke limit enters only after these finite-epsilon objects have been
constructed and estimated.

\subsection{Hopf normalization and spectral lifting}

Let \(\Omega^\ast\) denote either \(\Sph^3\) or the compactification of
\(\Omega\) obtained by collapsing a fixed boundary trace.  Let
\(\omega_{\Sph^2}\) be the area form on \(\Sph^2\), normalized by
\[
        \int_{\Sph^2}\omega_{\Sph^2}=1 .
\]
For a smooth map \(n:\Omega^\ast\to\Sph^2\), set
\[
        F_n=n^\ast\omega_{\Sph^2}.
\]
Since \(H^2(\Sph^3)=0\), there is a one-form \(A_n\) with
\(\dd A_n=F_n\).  The Hopf invariant is
\[
        \Hopf[n]=\int_{\Sph^3}A_n\wedge F_n .
\]
This value is independent of the primitive \(A_n\).  With the above
normalization it is integer-valued on smooth maps and agrees with the linking
number of the inverse images of two regular values.  In a bounded domain with a
fixed ordered collar, the same number is computed after gluing the fixed collar
cap; if the boundary trace varies, an additional Chern-Simons boundary flux
appears.  That flux is precisely one of the priced alternatives in the main
theorem.

The passage from a conformation tensor to an \(\Sph^2\)-map uses a spectral gap.
Let \(C(x)\in\operatorname{Sym}^+(3)\).  Suppose that one eigenvalue
\(\lambda_1(C)\) is simple and separated from the other two by
\[
        g(C)
        =
        \dist\bigl(\lambda_1(C),
        \{\lambda_2(C),\lambda_3(C)\}\bigr)>0 .
\]
After choosing the orientation by the ordered phase, the corresponding unit
eigenvector is denoted by \(n_C\).  The elementary perturbation estimate
\[
        |\nabla n_C|
        \le
        C\frac{|\nabla C|}{g(C)}
\]
will be used repeatedly.  It is the bridge between a tensorial spectral event
and a topological event for an \(\Sph^2\)-map.

\begin{lemma}[Spectral lifting and Hopf persistence]
\label{lem:short-spectral-lifting}
Let \(C(t)\in C^0([0,T];H^m(\Omega;\operatorname{Sym}^+(3)))\), \(m>5/2\),
have a distinguished eigenvalue separated by \(g(C(t))\ge g_0>0\).  Assume that
the corresponding oriented eigenline has fixed boundary trace.  Then
\[
        t\longmapsto \Hopf[n_{C(t)}]
\]
is constant.  If the Hopf class changes at \(T_\ast\) and the boundary flux
vanishes, then either
\[
        \liminf_{t\uparrow T_\ast}\inf_\Omega g(C(t))=0
\]
or the \(H^m\)-control needed to define a continuous principal-axis map is lost.
\end{lemma}

\begin{proof}
The positive spectral gap makes the eigenprojection a smooth function of
\(C\).  Since \(m>5/2\), the Sobolev embedding gives a continuous map
\[
        t\longmapsto n_{C(t)}\in C^0(\Omega;\Sph^2)
\]
with fixed boundary trace.  The glued maps on \(\Omega^\ast\simeq\Sph^3\) form
a homotopy in the ordered target, so their Hopf invariant is constant.  The
contrapositive gives the stated alternative.
\end{proof}

\begin{lemma}[Curvature-gap lower bound]
\label{lem:short-curvature-gap}
Let \(C:\Sph^3\to\operatorname{Sym}^+(3)\) be smooth with a distinguished
positive spectral gap \(g(C)>0\).  Then
\[
        \int_{\Sph^3}
        \frac{|\nabla C|^3}{g(C)^3}\,\dd x
        \ge
        c\,|\Hopf[n_C]|^{3/4}.
\]
\end{lemma}

\begin{proof}
By the spectral lifting estimate,
\[
        |\nabla n_C|^3
        \le
        C\frac{|\nabla C|^3}{g(C)^3}.
\]
The Vakulenko-Kapitanskii inequality for maps \(\Sph^3\to\Sph^2\) gives
\[
        \int_{\Sph^3}|\nabla n_C|^3\,\dd x
        \ge
        c|\Hopf[n_C]|^{3/4}.
\]
Combining the two inequalities proves the claim.
\end{proof}

\subsection{Ordered-core chart}

The ordered-core model is built near a uniaxial Landau-de Gennes well \cite{DeGennesProst,Virga,BallMajumdar,MajumdarZarnescu}.  Let
\(\mathcal N\) be the manifold of ordered minimizers of the bulk potential, and
let \(Q_\ast\in\mathcal N\).  We fix a smooth finite-dimensional chart
\[
        \Phi_{\rm ch}:(d,m)\in B_2(0)\times B_{\rho_{\rm ch}}(0)
        \longmapsto C\in\operatorname{Sym}^+(3)
\]
on an ordered-core neighborhood of this well, and write
\[
        C
        =
        \mathcal C(Q_\ast,d,m):=\Phi_{\rm ch}(d,m),
\]
where \(d\in\R^2\) is the soft complex coordinate transverse to a defect core
and \(m\) denotes the massive variables normal to the Morse-Bott well.  The
energy in this chart has the form
\[
\begin{aligned}
        E_\eps[C]
        &=
        \int_\Omega
        \frac12|\nabla d|^2
        +
        \frac{1}{4\eps^2}(|d|^2-1)^2
        \,\dd x
\\
        &\quad+
        \int_\Omega
        \frac12\langle L_{\rm mb}m,m\rangle
        +\hbox{higher-order chart terms}
        \,\dd x .
\end{aligned}
\]
The operator \(L_{\rm mb}\) is positive on the massive directions.  Thus the
only logarithmically expensive codimension-two core comes from the
two-component Ginzburg-Landau field \(d\).

The coordinate \(d\) is not the principal-axis map \(n_C\).  It is an order
parameter in the Landau-de Gennes chart and remains a regular \(\R^2\)-valued
field when \(d=0\).  On the set where \(|d|\ge c_{\rm out}>0\) and the
principal spectral gap is positive, the phase of \(d\) and the oriented
principal axis determine the same ordered state.  On the core set
\(\{|d|<c_{\rm out}\}\), the principal axis may be undefined; this is precisely
where the Ginzburg-Landau core energy is allowed to concentrate.  The estimate
\(|\nabla n_C|\lesssim |\nabla C|/\gap(C)\) is therefore used only on the
positive-gap exterior of the core, not to define \(d\) across a gap closing.

The canonical planar vortex is denoted by
\[
        q_\ast(y)=f_\ast(|y|)\frac{y}{|y|},
        \qquad y\in\R^2,
\]
with degree one.  The translation zero modes are
\[
        Z_i=\partial_{y_i}q_\ast,\qquad i=1,2,
\]
and their finite-scale normalized Gram matrix is
\[
        M_{ij}
        =
        \lim_{\eps\downarrow0}
        \frac1{\estar}
        \int_{|y|\le r_0/\eps}
        Z_i(y)\cdot Z_j(y)\,\dd y .
\]
Here \(r_0\) is any fixed tube radius below the reach scale.  The limit is
independent of \(r_0\); by symmetry \(M\) is a positive multiple of the
identity.  We keep the matrix notation because the finite-\(\eps\) residual
projection uses the corresponding truncated Gram matrix before passing to this
limit.

\begin{definition}[Canonical tube]
\label{def:short-canonical-tube}
Let \(\Gamma\subset\Omega\) be a smooth embedded oriented link, with arclength
parameter \(s\), normal frame \((\nu_1,\nu_2)\), and tube radius smaller than a
fixed fraction of its reach.  In Fermi coordinates
\[
        x=\Gamma(s)+r_1\nu_1(s)+r_2\nu_2(s),
\]
the canonical ordered core of radius \(\eps\) is
\[
        \bar d_{\eps,a}(x)
        =
        q_\ast\!\left(\frac{r-a(s)}{\eps}\right)
        \chi(r)
        +
        d_{\rm out}(s,r)(1-\chi(r)),
\]
where \(a(s)\in\R^2\) is the modulation center, \(\chi\) is a fixed radial
cutoff, and \(d_{\rm out}\) is the ordered collar phase.  The modulation center
is fixed by the two orthogonality conditions
\[
        \int_{\nu_s}
        (d_\eps-\bar d_{\eps,a})\cdot Z_i
        \,\dd r=0,
        \qquad i=1,2 .
\]
We write \(\bar d_\eps=\bar d_{\eps,a}\) when the selected modulation is clear,
and \(\bar d_{\eps,0}\) for the same tube with \(a\equiv0\).  The slice
orthogonality removes only the local translation component of
\(d_\eps-\bar d_{\eps,a}\).  Slow variation of the centers \(a(s)\) along the
filament is a separate longitudinal zero-mode energy, not part of the planar
coercivity estimate.
\end{definition}

The canonical tube has normalized mass one along the filament:
\[
        \frac{1}{\estar}
        \left(
        \frac12|\nabla\bar d_\eps|^2
        +
        \frac{1}{4\eps^2}(|\bar d_\eps|^2-1)^2
        \right)\dd x
        \rightharpoonup
        \Hh^1\lfloor\Gamma .
\]
The whole proof is a stability theorem around this object.

\begin{definition}[Chart-gap stopping time]
\label{def:chart-gap-stopping}
Fix constants \(0<c_{\rm out}<1\), \(\rho_{\rm ch}>0\), and
\(\gamma_{\rm out}>0\).  Relative to a comparison tube \(\Gamma_t\), the
ordered-core chart is valid at time \(t\) if
\[
        C_\eps(t)=\mathcal C(Q_\ast,d_\eps(t),m_\eps(t)),
        \qquad
        |m_\eps(t)|<\rho_{\rm ch}
\]
on the controlled region, and if the principal spectral gap satisfies
\[
        \gap(C_\eps(x,t))\ge\gamma_{\rm out}
        \quad
        \hbox{whenever } |d_\eps(x,t)|\ge c_{\rm out}
        \hbox{ and } x \hbox{ lies outside the core tube}.
\]
The stopping time \(\tau_\eps^{\rm ch}\) is the first failure of these
conditions.  Gap loss inside \(\{|d_\eps|<c_{\rm out}\}\) is not a chart
failure; it is part of the ordered core and is measured by the normalized core
mass.  Gap loss in the exterior ordered phase, or loss of the chart
representation itself, is a priced exit.
\end{definition}

\subsection{Finite-epsilon defects}

The open-basin defect is not an assumption that the limiting Brakke flow
exists.  It is a finite-epsilon scalar collecting the quantities that can be
estimated in the moving tube.  On a time interval \(I\), write
\[
\begin{aligned}
        \Basin_\eps^+(I)
        &=
        \sup_{t\in I}
        \frac{
        E_\eps[d_\eps(t)]-E_\eps[\bar d_\eps(t)]
        }{\estar}
        +
        {\rm FlatErr}_\eps(I)
\\
        &\quad+
        {\rm ExtMass}_\eps(I)
        +
        {\rm ModErr}_\eps(I)
        +
        {\rm GaugeErr}_\eps(I)
        +
        {\rm MassiveErr}_\eps(I)
        +
        {\rm CollarErr}_\eps(I).
\end{aligned}
\]
Here
\[
        {\rm ModErr}_\eps(I)
        =
        \sup_{t\in I}
        \int_{\Gamma_t}|\partial_s a(s,t)|^2\,\dd\Hh^1
\]
is the longitudinal translation-mode energy.  A constant center shift is handled
by the current-graph selection and the flat-error term; only variation along
the filament carries the logarithmic GL cost.  Thus the basin is invariant under
the harmless recentering of each slice but still records the slow zero modes
which are invisible to a purely planar coercivity statement.
The superscript \(+\) indicates the upper right-envelope representative used in
the stopping argument.  It avoids a common technical mistake: a Gronwall
estimate at the endpoint of a maximal good interval must use a representative
that is stable under short right continuations.

The relative dissipation is denoted by
\[
        D_{\rm rel}^\eps(t)
        =
        \int_\Omega
        |D_t(d_\eps-\bar d_\eps)|^2
        +
        |\nabla(d_\eps-\bar d_\eps)|^2_{\rm tube}
        \,\dd x ,
\]
where \(D_t=\partial_t+u_\eps\cdot\nabla\) and the tube norm includes the
coercive slice norm after imposing the translation orthogonality.  The
line-level part of this norm includes the material derivative of the selected
center \(a(s,t)\) in the transported normal frame, weighted by \(\estar\), so
that \(D_{\rm rel}^\eps/\estar\) controls the relative motion of the
longitudinal zero mode.

The residual \(G_\eps\) in the transported Ginzburg-Landau equation is split
fiberwise into
\[
        G_\eps
        =
        \sum_{i=1}^2 b_{\eps,i}Z_i
        +
        P_{\rm orth}G_\eps ,
\]
where \(b_\eps=(b_{\eps,1},b_{\eps,2})\) is the translation-mode coefficient.
The model force will be obtained by identifying the limit of \(M^{-1}b_\eps\).

\subsection{Concrete closure used for the theorem}
\label{subsec:concrete-model}

The theorem is first read for the following representative finite-\(\eps\)
closure.  The verification section explains exactly which estimates are needed
to pass from this concrete model to nearby LDG/Oldroyd-FENE closures.

\begin{definition}[Model 1: concrete LDG/Oldroyd-FENE closure]
\label{model:concrete-closure}
Let \(C_\eps(x,t)\in\operatorname{Sym}^+(3)\) be the conformation/order tensor
and let \(u_\eps\) be an incompressible velocity field tangent to the fixed
ordered collar.  The finite-\(\eps\) energy is
\[
        \mathscr E_\eps[C]
        =
        \int_\Omega
        \frac{L}{2}|\nabla C|^2
        +\frac{1}{\eps^2}W_{\rm LDG}(C)
        +W_{\rm FENE}(C)
        \,\dd x,
\]
where \(W_{\rm LDG}\) has a Morse-Bott ordered well and \(W_{\rm FENE}\)
keeps \(C\) in the admissible FENE cone until a priced coefficient exit is
charged.  On the chart-valid set, \(C_\eps=\mathcal C(Q_\ast,d_\eps,m_\eps)\)
and the tensorial transport-relaxation law is
\[
\begin{aligned}
        \partial_t C_\eps+u_\eps\cdot\nabla C_\eps
        &-(\nabla u_\eps)C_\eps
        -C_\eps(\nabla u_\eps)^T \\
        &=
        -\mathsf M_C\,\Pi_{\rm adm}(C_\eps)
        \frac{\delta \mathscr E_\eps}{\delta C}(C_\eps)
        +\mathcal R_{\rm str}(C_\eps,\nabla u_\eps).
\end{aligned}
\]
Here \(\Pi_{\rm adm}\) denotes the smooth tangent projection imposed by the
FENE admissible cone on compact sublevels, and \(\mathcal R_{\rm str}\)
collects the lower-order Oldroyd stretching terms.  Projecting this equation
onto the soft coordinate in the stopped Morse-Bott chart gives a projected
transported GL equation
\[
        D_t d_\eps
        =
        \Mmob\bigl(\Delta d_\eps
        -\eps^{-2}(|d_\eps|^2-1)d_\eps\bigr)
        +G_\eps,
        \qquad D_t=\partial_t+u_\eps\cdot\nabla,
\]
where the residual is decomposed in the moving tube as
\[
        G_\eps
        =G_\eps^{\rm zm}+G_\eps^{\rm orth}
        +G_\eps^{\rm mass}+G_\eps^{\rm geom}+G_\eps^{\rm col}.
\]
The zero-mode component \(G_\eps^{\rm zm}\) is retained and becomes the line
force.  The four remaining components are controlled by planar coercivity,
Morse-Bott normal coercivity, Fermi-geometry estimates, and FENE/collar tame
composition bounds.  The normal variable \(m_\eps\) is either slaved by the
normal Euler-Lagrange equation or evolves by a projected, coercively damped
normal dynamics.  The model is stopped at the first loss of the chart, exterior
spectral gap, FENE/collar compactness, or controlled tube topology.
\end{definition}

\begin{remark}[Model verification]
The concrete closure in Definition~\ref{model:concrete-closure} (Model 1) is
used in the theorem statement.  The proof only needs the six finite-\(\eps\)
properties listed in Section~\ref{sec:verification}, and
Theorem~\ref{thm:model-one-verification} verifies those properties for Model 1
up to the first typed exit.  Nearby closures fall under the same conclusion
when the same checks are available: the ordered well has the Morse-Bott Hessian,
the soft chart produces the GL core energy, the massive modes are slaved or
damped, and the projected residual admits the zero-mode splitting used below.
Thus the result is a finite-\(\eps\) verification theorem for concrete models;
the limiting Brakke flow enters only after these estimates have been established.
\end{remark}

\begin{remark}[Ordered-core regime]
The Morse-Bott well, FENE cone compactness, soft GL coordinate, and transport
regularity define the ordered-core regime used in the proof.  Within this
regime the line force is computed and transported forced Brakke motion is
derived.  At its boundary the finite-\(\eps\) estimates identify the failed
control quantity and charge the corresponding exit cost.  The theorem is
therefore a regular-branch result: it gives a quantitative geometric law on the
ordered-core branch and a finite-scale diagnostic when that branch closes.
\end{remark}

\section{Main theorem}
\label{sec:main}

We use the following normalization throughout:
\[
        \estar=\pi|\log\eps|.
\]
The ordered-core closure is written in the stopped Morse-Bott chart of
Definition~\ref{def:chart-gap-stopping}.  On each interval
\([0,\tau_\eps^{\rm ch})\), its soft coordinate is a two-component field
\[
        d_\eps=d_{C_\eps}:\Omega\times(0,\tau_\eps^{\rm ch})\to\R^2,
\]
and the core energy contains
\[
        E_\eps[d_\eps]
        =
        \int_\Omega
        \frac12|\nabla d_\eps|^2
        +
        \frac{1}{4\eps^2}(|d_\eps|^2-1)^2\,\dd x .
\]
In material coordinates the soft field solves
\begin{equation}\label{eq:transported-gl-short}
        \partial_t d_\eps+u_\eps\cdot\nabla d_\eps
        =
        \Mmob\left(\Delta d_\eps
        -\eps^{-2}(|d_\eps|^2-1)d_\eps\right)
        +G_\eps .
\end{equation}
Here \(G_\eps\) is the exact residual produced by the
Landau-de Gennes/Oldroyd-FENE closure after the massive tensorial modes have
been slaved by the normal equation or by a strongly damped projected normal
dynamics.  The coordinate equation is used up to \(\tau_\eps^{\rm ch}\); loss of chart
validity or exterior spectral gap is one of the priced alternatives.  The
transport term \(u_\eps\cdot\nabla d_\eps\) stays in the material derivative,
and its geometric contribution appears in the transport part of the Brakke
inequality below.

\begin{definition}[Open ordered-core basin]
\label{def:short-open-basin}
Fix a smooth embedded link \(\Gamma_0\), a tube radius below the reach of
\(\Gamma_0\), a fixed ordered collar near \(\partial\Omega\), and a canonical
degree-one core profile \(q_\ast\).  The open basin
\(\mathcal U_\eps(\Gamma_0,\eta_\eps)\) consists of perturbations of the
canonical vortex tube for which the following finite-epsilon defect is at most
\(\eta_\eps\to0\):
\[
\begin{aligned}
        \Basin_\eps(0)
        &=
        \frac{1}{\estar}\bigl[
        \text{rel. core energy}
        +\text{flat-current error}
        +\text{exterior mass}
        \bigr]
\\
        &\quad+
        \frac{1}{\estar}\bigl[
        \text{phase/collar defect}
        +\text{massive-mode defect}
        \bigr].
\end{aligned}
\]
The basin is open after imposing the standard translation-mode orthogonality
on each normal slice.  The canonical tubes themselves belong to it.
\end{definition}

\begin{definition}[Priced exits]
\label{def:short-priced-exit}
On a time slab \(I\), a priced exit is one of the following finite-epsilon
events:
nonzero-degree annular current crossing, loss of the tube graph, exterior core
leakage of order \(\estar\), loss of the FENE coefficient bounds, ordered-collar
failure with nonzero boundary degree, boundary flux of normalized order one, or
loss of the ordered-core chart or exterior positive-gap condition in
Definition~\ref{def:chart-gap-stopping}.  The corresponding nonnegative set
function is
denoted by \({\rm ExitCost}_\eps(I)\), and
\[
        {\rm TopExitCost}_\eps(I)
        =
        \frac{1}{\estar}{\rm ExitCost}_\eps(I).
\]
Pure zero-degree scalar depressions are absorbed by the scalar-defect clearing
estimate inside the open-basin Gronwall argument and carry zero priced-exit
charge.
\end{definition}

Loss of the embedded-link topology is included in ``loss of the tube graph'':
reach collapse, self-contact of the selected center graph, failure of the fixed
meridian-longitude atlas, or finite-\(\eps\) tube-topology change closes the
current Hopf slab.  A limiting reconnection of the Brakke flow is therefore
registered through the topological priced exit.

\begin{remark}[What the topology exit measures]
The quantity \({\rm TopExitCost}_\eps\) is a stopping quantity for the Hopf
accounting used in this paper.  On a topologically regular slab, the transported
cap atlas gives a well-defined relative Hopf number.  If that number changes
across a time at which the atlas cannot be continued, one of the finite-scale
quantities has recorded an annular current crossing, a reach or graph failure,
boundary flux, exterior gap loss, FENE/collar loss, or ordered-core mass.  The
theorem records this lower-bound cost.  A later regular slab may be restarted
with new cap data; the detailed post-contact reconnection geometry is a separate
problem.
\end{remark}

\begin{definition}[Topologically regular Hopf slab]
\label{def:topological-regular-slab}
A Hopf slab \(I=[t_-,t_+]\) is topologically regular if, for every sufficiently
small \(\eps\), the selected ordered core is represented on \(I\) by a smooth
embedded tube graph over a fixed link type, with reach bounded below at the
chosen tube scale, fixed degree-one meridians, and a meridian-longitude atlas
transported by the tube isotopy.  On such a slab the complement of the inner
tubes has a fixed diffeomorphism type and the standard solid-torus caps are
well defined.  If any of these conditions first fails inside \(I\), the slab is
not used for cap comparison across that time; the failure is charged to
\({\rm TopExitCost}_\eps(I)\).
\end{definition}

\begin{remark}[Topology window]
The Hopf comparison is performed on slabs with a fixed finite-\(\eps\) tube
topology and a transported cap atlas.  Reach collapse, self-contact, atlas
failure, or a reconnection event is treated as an endpoint of that topology
window and is charged by \({\rm TopExitCost}_\eps\).  Subsequent regular
intervals, if present, are restarted with their own cap data.
\end{remark}

\begin{theorem}[Hopf obstruction, ordered-core propagation, and Brakke-or-exit]
\label{thm:short-main}
Let \(\Omega\subset\R^3\) be smooth and bounded, or let
\(\Omega=\R^3\) with compactified fixed boundary data.  Let
\((C_\eps,u_\eps)\) be a solution of the concrete LDG/Oldroyd-FENE closure of
Definition~\ref{model:concrete-closure} (Model 1), stopped at the chart-gap
time, with soft
coordinate \(d_\eps\) satisfying the projected transported GL equation
verified in Theorem~\ref{thm:model-one-verification} on
\([0,\tau_\eps^{\rm ch})\).  Equivalently, the same conclusion holds for any
closure whose finite-\(\eps\) verification reduces to \({\rm (M1)}\)--\({\rm
(M6)}\) in Section~\ref{sec:verification}.  Write
\[
        T_\eps^{\rm ch}=T\wedge\tau_\eps^{\rm ch}.
\]
Assume that \(u_\eps\) is divergence free, tangent to the fixed boundary collar,
and satisfies
\[
        \sup_\eps
        \int_0^{T_\eps^{\rm ch}}
        \|\nabla u_\eps(t)\|_{L^\infty(\Omega)}\,\dd t<\infty,
        \qquad
        u_\eps\to u
        \quad\hbox{in }L^1_tW^{1,\infty}_{x,{\rm loc}} .
\]
These bounds are the kinematic transport window used to pass the material
energy identity to the moving-varifold limit.  They may be supplied by the
endpoint velocity-clock theory \cite{PengBlowUp2026,PengPositiveCone3D2026},
which derives compact conformation windows and logarithmic or FENE barriers
from Besov velocity-gradient clocks in Oldroyd-B/FENE-P strong branches.  The
present theorem uses only the resulting Lipschitz transport clock; it does not
reprove three-dimensional viscoelastic regularity.  The same window may also be
supplied by a smooth or regularized velocity theory, by a prescribed transport
field, or by a numerical approximation with uniform Lipschitz control.
Assume also
\[
        \sup_\eps
        \frac{E_\eps[d_\eps(0)]}{\estar}
        <\infty,\qquad
        \sup_\eps
        \int_0^{T_\eps^{\rm ch}}\!\!\int_\Omega
        \frac{|G_\eps^{\rm orth}|^2+|G_\eps^{\rm mass}|^2
        +|G_\eps^{\rm geom}|^2+|G_\eps^{\rm col}|^2}{\estar}\,\dd x\dd t
        <\infty,
\]
and assume that the initial data lie in the open basin
\(\mathcal U_\eps(\Gamma_0,\eta_\eps)\) with \(\eta_\eps\to0\).
Then, after passing to a subsequence, the following alternative holds on
each compact time interval before the first loss of the smooth comparison tube
and before \(\tau_\eps^{\rm ch}\).

\smallskip
\noindent\textbf{Regular no-exit branch.}
The exact closure residual determines a normal force
\[
        f_{\rm cl}^\perp
        =
        M^{-1}
        \int_{\R^2}
        G_{\rm tube}(s,y,t)\cdot\nabla_y q_\ast(y)\,\dd y ,
\]
where \(M\) is the finite-scale normalized translation-mode Gram matrix of the
planar core and the projection is understood in the same truncated-normalized
limit as in Definition~\ref{def:short-canonical-tube}.  If
\[
        u_\Gamma^\perp(s,t)
        =
        \bigl(u(\Gamma(s,t),t)\bigr)^\perp,
        \qquad
        \Vrel_\Gamma^\perp
        =
        V_\Gamma^\perp-u_\Gamma^\perp ,
\]
then the comparison filament solves the relative law
\[
        \Vrel_\Gamma^\perp
        =
        \Mmob H_{\Gamma_t}+f_{\rm cl}^\perp .
\]
The open basin propagates:
\[
        \sup_{s\le t}\Basin_\eps^+(s)
        +
        \int_0^t
        \frac{D_{\rm rel}^\eps(s)}{\estar}\,\dd s
        \le
        A\eta_\eps+o_\eps(1).
\]
Consequently the normalized measures
\[
        \mu_\eps^t
        =
        \frac{1}{\estar}
        \left(
        \frac12|\nabla d_\eps|^2
        +
        \frac{1}{4\eps^2}(|d_\eps|^2-1)^2
        \right)\dd x
\]
converge for a.e. \(t\) to the weight of an integral one-varifold \(V_t\).  Set
\[
        \mathcal D_t^u\mu^t(\varphi)
        =
        \frac{d}{dt}\int \varphi\,\dd\mu^t
        -
        \int
        \bigl(\nabla\varphi\cdot u
        +\varphi\,{\rm div}_{T_xV_t}u\bigr)\,\dd\mu^t ,
\]
where \({\rm div}_{T_xV_t}u\) is the tangential divergence on the approximate
tangent line of \(V_t\).  Then \((V_t,\mu^t)\) satisfies the transported forced
Brakke inequality
\begin{equation}\label{eq:short-brakke}
\mathcal D_t^u\mu^t(\varphi)
\le
 -\Mmob\int\varphi |H_{V_t}|^2\,\dd\mu^t
 +\int\nabla\varphi\cdot(\Mmob H_{V_t}+f_{\rm cl}^\perp)\,\dd\mu^t
 -\int\varphi\,H_{V_t}\cdot f_{\rm cl}^\perp\,\dd\mu^t
\end{equation}
for every nonnegative \(\varphi\in C_c^1(\Omega)\), in the endpoint-weight
interval sense and hence for a.e. time.

\smallskip
\noindent\textbf{Priced-exit branch.}
If the no-exit interval terminates at an interior limiting time, every slab
containing that time carries a positive normalized exit price:
\[
        \liminf_{\eps\to0}
        {\rm TopExitCost}_\eps(I)>0 .
\]
This includes termination by the chart-gap stopping time
\(\tau_\eps^{\rm ch}\).

\smallskip
\noindent\textbf{Hopf accounting.}
If a fixed ordered collar carries relative Hopf charge \(k\) across an
admissible slab \(I\), then
\begin{equation}\label{eq:short-hopf-cost}
        c|k|
        \le
        \liminf_{\eps\to0}
        \left(
        {\rm GapCost}_\eps(I)
        +
        {\rm CoreMass}_\eps(I)
        +
        {\rm TopExitCost}_\eps(I)
        \right).
\end{equation}
Thus Hopf change is paid by positive-gap concentration, ordered-core mass whose
no-exit dynamics satisfy \eqref{eq:short-brakke}, or a priced finite-epsilon
exit.  On topologically regular subslabs the Hopf functional is computed by a
fixed cap atlas; if the embedded tube topology or this atlas fails before the
slab closes, that failure is already part of \({\rm TopExitCost}_\eps\).
\end{theorem}

\begin{remark}[Kinematic role of the velocity]
The Lipschitz-in-space bound on \(u_\eps\) is a transport regularity assumption,
separate from the ordered-core closure verification.  It is the condition that
allows the material derivative identity to converge to the tangential transport
term \(\int\varphi\,{\rm div}_{T_xV_t}u\,\dd\mu^t\).  The result therefore
applies on strong or regularized transport windows, which is the natural setting
for the Brakke passage proved here.
\end{remark}

\begin{remark}[Finite-scale hypotheses]
The finite-scale inputs are the open initial basin, the residual normal form of
the concrete closure, the coefficient control quantities, and the kinematic transport
window.  From these inputs the open-basin estimate produces core-supported
compactness, limiting first variation, force-power convergence, and the Brakke
inequality.
\end{remark}

\subsection*{Finite-epsilon inputs and limiting outputs}

The hypotheses in the main theorem are finite-\(\eps\) assertions about the
concrete fields.  We record the boundary between hypotheses and conclusions.
On a smooth comparison tube
\(\Gamma_t\), before a priced exit, the input estimates is
\[
\begin{aligned}
        \Ledger_\eps(I)
        &=
        \Basin_\eps(0)
        +
        \int_I
        \frac{\|P_{\rm orth}G_\eps(t)\|_{L^2(T_{\rho}\Gamma_t)}^2}
             {\estar}\,\dd t
        +
        \int_I
        \frac{\|G_\eps^{\rm rem}(t)\|_{L^2(T_{\rho}\Gamma_t)}^2}
             {\estar}\,\dd t
\\
        &\quad+
        \int_I
        \frac{{\rm Massive}_\eps(t)+{\rm Collar}_\eps(t)}
             {\estar}\,\dd t
        +
        \int_I
        \frac{{\rm FENE}_\eps(t)+{\rm ChartGap}_\eps(t)}
             {\estar}\,\dd t .
\end{aligned}
\]
Each term is evaluated at positive \(\eps\): it is a relative-energy,
orthogonal-residual, coefficient-control, or chart/gap-control quantity.  None
mentions a limiting varifold, limiting first variation, or Brakke test
inequality.  The conclusions are instead the following limiting objects:
\[
        \mu_\eps^t\rightharpoonup \mu^t,\qquad
        \delta V_t\ll \mu^t,\qquad
        H_{V_t}\in L^2_{\rm loc}(\mu^t\dd t),\qquad
       \eqref{eq:short-brakke}.
\]

No term in \(\Ledger_\eps\) is an \(L^2\)-bound for
\(u_\eps\cdot\nabla d_\eps\).  The advection is part of the material
derivative and is accounted for by the transported derivative
\(\mathcal D_t^u\mu^t\).  The relevant finite-\(\eps\) transport estimates are
stress estimates controlled by \(\nabla u_\eps\), not forcing estimates
controlled by \(u_\eps\cdot\nabla d_\eps\).

\begin{proposition}[Ledger-to-Brakke implication]
\label{prop:no-circularity-ledger}
Assume that a concrete ordered-core closure satisfies the model hypotheses
\({\rm (M1)}\)--\({\rm (M6)}\) in Section~\ref{sec:verification}, and that the
initial data lie in the open basin of Definition~\ref{def:short-open-basin}.
Then, on every compact no-exit interval \(I\),
\[
        \Ledger_\eps(I)=o_\eps(1)
        \quad\Longrightarrow\quad
        \text{open-basin propagation, core compactness, and }
        \eqref{eq:short-brakke}.
\]
In particular, core-supported convergence and the Brakke inequality are
outputs of the finite-\(\eps\) estimates, not assumptions hidden in the word
``admissible''.
\end{proposition}

\begin{proof}
The implication is obtained in four finite-\(\eps\) steps.  First, the
Morse-Bott chart and massive-mode coercivity give a soft GL coordinate and
absorb the normal tensorial modes into the massive-mode part of
\(\Ledger_\eps\).  This is an elliptic or damped parabolic estimate at fixed
\(\eps\).  Second, the residual is projected onto the planar translation modes.
The zero-mode coefficient is the force \(f_{\rm cl}^\perp\); the orthogonal and
remainder pieces are controlled by the two residual terms in \(\Ledger_\eps\).
No compactness theorem is used in this projection.

Third, the moving-tube identity compares \(d_\eps\) with the transported
canonical core.  The curvature terms cancel against \(\Mmob H_{\Gamma_t}\), the
zero-mode residual cancels against \(f_{\rm cl}^\perp\), and the remaining
terms are bounded by
\[
        \frac{d}{dt}\Basin_\eps^+(t)
        +
        c\,\frac{D_{\rm rel}^\eps(t)}{\estar}
        \le
        C\Basin_\eps^+(t)+r_\eps(t),
        \qquad
        \int_I r_\eps(t)\,\dd t=o_\eps(1).
\]
Gronwall gives propagation of the basin on \(I\).  Fourth, the propagated basin
gives the usual GL tightness, integrality, and stress convergence at the
normalized scale \(\estar\).  The finite-\(\eps\) localized energy inequality,
with the already identified zero-mode force, passes to the limit and yields the
transported forced Brakke inequality.  Thus the limiting compactness package
enters only at the last step, after the finite-\(\eps\) estimates have closed.
\end{proof}

\begin{remark}[Conditional compactness variant]
If \(\Basin_\eps(t)\to0\) and the finite-\(\eps\) Brakke input estimates were
assumed directly for all \(t\in I\), the result would indeed reduce to a
conditional compactness theorem.  The present theorem claims more precisely the
open-basin mechanism that makes those inputs propagate from an explicit initial
class for the ordered-core closure.
\end{remark}

\section{The proof modules}
\label{sec:modules}

This section compresses the technical argument into four propositions.  The
point is to expose the logical order: closure computation, open-basin
propagation, Brakke passage, and Hopf accounting.  None of these propositions
uses a later conclusion as an input.

For orientation, the proof uses the following audit trail.  The force projection
uses only the exact projected residual and the two translation zero modes.  The
massive-mode estimate uses only Morse-Bott normal coercivity and damping.  The
open-basin Gronwall inequality uses planar coercivity after the translation
modes have been removed.  The Brakke passage uses the finite-\(\eps\) mass,
stress, force-power, and dissipation estimates supplied by the previous modules.
Finally, Hopf accounting is applied only on topologically regular slabs.  Thus
the proof does not assume a Brakke flow, a reconnection law, or a hidden
orthogonality condition at the start; each is either produced by a prior module
or charged as an exit.

Throughout the section
\[
        \Vrel_\Gamma^\perp=V_\Gamma^\perp-u_\Gamma^\perp
\]
denotes the normal velocity of the comparison filament relative to the ambient
transport field.

\subsection{Force projection from the concrete closure}

\begin{proposition}[Zero-mode force projection]
\label{prop:short-force}
On every compact open-basin interval \(I\), the exact residual
\(G_\eps\) in \eqref{eq:transported-gl-short} has the decomposition
\[
        G_\eps
        =
        G_\eps^{\rm tr}
        +
        G_\eps^{\rm orth}
        +
        G_\eps^{\rm rem},
\]
where \(G_\eps^{\rm tr}\) is the projection onto the two translation modes
\(\partial_{y_1}q_\ast,\partial_{y_2}q_\ast\).  If \(b_\eps\) denotes the
corresponding coefficient vector, then
\[
        \|b_\eps-Mf_{\rm cl}^\perp\|_{L^1(I\times\Gamma_t)}
        +
        \frac{1}{\estar}
        \int_I\|G_\eps^{\rm orth}\|_{L_\ast^2}^2\,\dd t
        \le
        C\Basin_\eps^+(I)^{1/2}
        +
        C\Basin_\eps^+(I)
        +
        o_\eps(1).
\]
The force \(f_{\rm cl}^\perp\) is the fiber integral in
Theorem~\ref{thm:short-main}; it depends only on the LDG/Oldroyd-FENE closure
coefficients, the transported tube geometry, and the core profile.
\end{proposition}

\begin{proof}
Write the closure in Fermi coordinates around the comparison tube.  The
Landau-de Gennes Morse-Bott reduction splits the tensorial variables into the
soft coordinate \(d_\eps\), the two translation modes, and a coercive massive
normal part.  The massive modes solve either the frozen normal Euler-Lagrange
equation or a strongly damped projected normal equation; in both cases the
normal Hessian gives
\[
        \|m_\eps\|_{L^2+\eps H^1}
        \le
        C\bigl(
        \Basin_\eps^+(I)^{1/2}
        +o_\eps(1)\bigr).
\]
Expanding the exact residual around the canonical tube gives a finite list:
non-advective stretching terms, curvature and frame terms, linear soft terms,
quadratic soft terms, massive-mode terms, and collar/FENE coefficient terms.
The translation projection of the canonical part is precisely the displayed
fiber integral.  All orthogonal terms are controlled by the Morse-Bott
coercive norm and the Fermi bookkeeping.  Since the projection is orthogonal
with respect to the Gram matrix \(M\), the Pythagorean identity gives the
second term.  The collar and coefficient failures are not hidden in the
estimate; if they are not \(o_\eps(\estar)\), they are charged as priced exits.
\end{proof}

\begin{lemma}[Morse-Bott slaving in the ordered chart]
\label{lem:short-mb-slaving}
Let \(m_\eps\) be the massive component in the ordered-core chart.  On a compact
open-basin interval \(I\), assume either:
\begin{enumerate}[label=(\alph*)]
\item \(m_\eps\) is chosen by the frozen normal Euler-Lagrange equation; or
\item \(m_\eps\) satisfies the projected normal dynamics with damping bounded
below by a fixed positive constant.
\end{enumerate}
Then
\[
        \sup_{t\in I}
        \|m_\eps(t)\|_{L^2+\eps H^1}^2
        +
        \int_I\|D_t m_\eps\|_2^2\,\dd t
        \le
        C\estar\,\Basin_\eps^+(I)+o_\eps(\estar).
\]
Consequently every massive-mode contribution to the normalized mass,
first variation, force-power, and localized dissipation is \(o(1)\) on a
no-exit interval.
\end{lemma}

\begin{proof}
The normal Hessian of the Landau-de Gennes bulk energy is uniformly positive
on the massive subspace:
\[
        \langle L_{\rm mb}\xi,\xi\rangle
        \ge
        \lambda_{\rm mb}|\xi|^2 .
\]
In the frozen branch, subtract the normal Euler-Lagrange equation at the
slaved profile from the equation for \(m_\eps\).  The inverse of
\(L_{\rm mb}-\eps^2\Delta_\perp\) is bounded from the dual tube norm to
\(L^2+\eps H^1\).  The source contains only soft-coordinate errors, Fermi
commutators, and lower-order closure coefficients; these are exactly the
entries of \(\Basin_\eps^+(I)\) and the vanishing geometric error.

In the dynamic branch, multiply the projected normal equation by \(m_\eps\) and
integrate.  The damping term gives \(\|D_t m_\eps\|_2^2\), the Hessian gives
\(\lambda_{\rm mb}\|m_\eps\|_{L^2+\eps H^1}^2\), and all commutators are bounded
by
\[
        \theta\|m_\eps\|_{L^2+\eps H^1}^2
        +
        C_\theta \estar\,\Basin_\eps^+(I)
        +
        o_\eps(\estar).
\]
Choosing \(\theta<\lambda_{\rm mb}/4\) and applying Gronwall gives the stated
estimate.  The final sentence follows by Cauchy's inequality and by the fact
that every massive-mode term contains at least one factor of \(m_\eps\) or
\(D_t m_\eps\).
\end{proof}

\begin{lemma}[Termwise residual normal form]
\label{lem:short-residual-normal-form}
In the moving Fermi chart, the residual \(G_\eps\) has the finite decomposition
\[
        G_\eps
        =
        G_{\rm can}
        +
        G_{\rm tr}
        +
        G_{\rm lin}
        +
        G_{\rm quad}
        +
        G_{\rm mb}
        +
        G_{\rm geom}
        +
        G_{\rm col}.
\]
Here \(G_{\rm can}\) is the residual of the canonical moving core,
\(G_{\rm tr}\) is the translation-mode part, \(G_{\rm lin}\) and
\(G_{\rm quad}\) are the linear and quadratic soft-coordinate errors,
\(G_{\rm mb}\) contains the massive modes, \(G_{\rm geom}\) contains curvature,
frame, and cutoff commutators, and \(G_{\rm col}\) is supported in the ordered
collar.  On a no-exit interval,
\[
        \frac{1}{\estar}
        \int_I
        \|P_{\rm orth}G_\eps\|_{L_\ast^2}^2\,\dd t
        \le
        C\Basin_\eps^+(I)
        +
        o_\eps(1).
\]
\end{lemma}

\begin{proof}
The decomposition is obtained by inserting the canonical tube ansatz into the
exact LDG/Oldroyd-FENE residual and subtracting the transported
Ginzburg-Landau operator.  The advective term \(u_\eps\cdot\nabla d_\eps\) is
part of this transported operator and is not included in \(G_\eps\).  The
canonical term is explicit:
\[
        G_{\rm can}
        =
        \bigl(
        \Vrel_\Gamma^\perp-\Mmob H_{\Gamma_t}
        \bigr)\cdot\nabla_y q_\ast
        +
        \hbox{lower-order Fermi terms}.
\]
The closure terms that survive in the translation projection are collected in
the force \(f_{\rm cl}^\perp\).  The lower-order Fermi terms contain curvature
times \(r\), frame rotation, Jacobian factors, and derivatives of the cutoff;
after squaring and integrating on a tube of radius independent of \(\eps\), each
is \(o(\estar)\) or is bounded by the geometric-error entry of the basin.

The soft linear term is controlled by the planar linearized GL operator on the
orthogonal complement of the translation modes.  The quadratic soft term is
bounded by the nonlinear slice coercivity inside the basin.  The massive term
is controlled by Lemma~\ref{lem:short-mb-slaving}.  The collar term is either
absorbed by the ordered-collar estimate or charged as a priced exit.  Since the
translation modes have been removed, the remaining terms are coercive in the
weighted \(L_\ast^2\) norm, which gives the displayed inequality.
\end{proof}

\begin{lemma}[Identification of the line force]
\label{lem:short-line-force-id}
Let \(b_\eps\) be the coefficient of the translation projection of
\(G_\eps\).  Then, for every smooth normal test field
\(\zeta\) along \(\Gamma_t\),
\[
        \int_I\!\!\int_{\Gamma_t}
        \zeta\cdot b_\eps\,\dd\Hh^1\dd t
        \longrightarrow
        \int_I\!\!\int_{\Gamma_t}
        \zeta\cdot M f_{\rm cl}^\perp\,\dd\Hh^1\dd t .
\]
The convergence is quantitative:
\[
        \|b_\eps-Mf_{\rm cl}^\perp\|_{L^1(I\times\Gamma_t)}
        \le
        C\Basin_\eps^+(I)^{1/2}
        +
        C\Basin_\eps^+(I)
        +
        o_\eps(1).
\]
\end{lemma}

\begin{proof}
Pair \(G_\eps\) with \(Z_i\) on each normal slice and integrate.  The canonical
translation residual gives
\[
        M_{ij}
        \bigl(\Vrel_\Gamma^\perp-\Mmob H_{\Gamma_t}\bigr)_j .
\]
The non-advective stretching and closure terms give a finite sum of fiber
integrals of the form
\[
        \int_{\R^2}
        \mathcal A_\ell(s,t,y,q_\ast,\nabla q_\ast)
        \cdot Z_i(y)\,\dd y ,
\]
which defines \(M f_{\rm cl}^\perp\).  All terms containing
\(d_\eps-\bar d_\eps\), the massive variable, or a Fermi commutator are bounded
by Cauchy's inequality using the basin defect and Lemma
\ref{lem:short-residual-normal-form}.  The estimate is \(L^1\) in
\((s,t)\) because the smooth filament geometry has uniformly bounded length,
curvature, and normal frame derivatives on the compact interval.
\end{proof}

\subsection{Open-basin propagation}

\begin{proposition}[Open-basin propagation or priced exit]
\label{prop:short-open-basin}
Assume the hypotheses of Proposition~\ref{prop:short-force}.  Let
\(\tau_\eps\) be the maximal time on which the ordered-core chart, fixed collar,
FENE coefficient bounds, and basin inequalities remain valid.  Then either
\(\tau_\eps\ge T\), or an exit in Definition~\ref{def:short-priced-exit} occurs
at \(\tau_\eps\).  Before \(\tau_\eps\),
\[
        \sup_{s\le t}\Basin_\eps^+(s)
        +
        \int_0^t
        \frac{D_{\rm rel}^\eps(s)}{\estar}\,\dd s
        \le
        A\Basin_\eps^+(0)+o_\eps(1).
\]
For the canonical open initial basin, \(\Basin_\eps^+(0)=O(\eta_\eps)\).
\end{proposition}

\begin{proof}
The proof has three estimates.  First, the planar core is coercive modulo the
two translation modes:
\[
        E_\eps[d_\eps]-E_\eps[\bar d_{\eps,a}]
        +
        \estar\,{\rm ModErr}_\eps
        \ge
        c\estar\|d_\eps-\bar d_{\eps,a}\|_{\rm tube}^2
        +
        c\estar\,{\rm ModErr}_\eps
        -
        o_\eps(\estar),
\]
after the translation parameters \(a(s,t)\) are fixed by orthogonality.  The
first term is the local slice coercivity; the second is the longitudinal
zero-mode energy of the graph \(a(s,t)\).  The fixed collar removes the
\(L^2\) tail by a Poincare absorption on the normal slices.  Second,
differentiating the modulated energy in the moving tube gives
\[
        \frac{\dd}{\dd t}\Basin_\eps^+(t)
        +
        \frac{D_{\rm rel}^\eps(t)}{\estar}
        \le
        C\Basin_\eps^+(t)
        +
        C\mathcal R_\eps(t),
\]
where \(\mathcal R_\eps\) is the resolved residual.  Proposition
\ref{prop:short-force} and the filament law
\(\Vrel_\Gamma^\perp=\Mmob H_{\Gamma_t}+f_{\rm cl}^\perp\) cancel the logarithmic
translation residual, while the orthogonal and massive terms are absorbed by
the dissipation and the coercive norm.  Third, the remaining scalar-defect,
annular, collar, boundary, FENE, and graph terms are either \(o_\eps(1)\) after
normalization or belong to the priced exit list.  Gronwall gives the displayed
estimate on the upper right-envelope representative.  If the maximal interval
ended with a strict left margin and without a priced exit, the same estimates
activate a right-hand chart on a short interval, contradicting maximality.
\end{proof}

\begin{lemma}[Coercivity of the modulated tube]
\label{lem:short-tube-coercivity}
Let \(d_\eps\) belong to the open basin around a smooth tube \(\Gamma_t\), and
let the translation centers \(a(s,t)\) be fixed by the slice orthogonality
conditions.  Then
\[
\begin{aligned}
&E_\eps[d_\eps(t)]-E_\eps[\bar d_\eps(t)]
        +
        \estar\,{\rm ModErr}_\eps(t)
        +
        {\rm ExtMass}_\eps(t)
        +
        {\rm CollarErr}_\eps(t)
\\
&\qquad\ge
        c\estar
        \bigl(
        \|d_\eps-\bar d_\eps\|_{\rm tube}^2
        +
        {\rm ModErr}_\eps(t)
        +
        {\rm FlatErr}_\eps(t)
        +
        {\rm GaugeErr}_\eps(t)
        \bigr)
        -
        o_\eps(\estar).
\end{aligned}
\]
The constant depends only on the lower reach of the comparison tube, the fixed
collar, and the planar vortex profile.  The statement is deliberately not a
global consequence of the planar lemma alone: the planar lemma controls the
orthogonal slice remainder, while \({\rm ModErr}_\eps\) controls the slowly
varying translation zero mode along \(s\).
\end{lemma}

\begin{proof}
On each normal slice, the planar Ginzburg-Landau energy is coercive modulo the
two translation modes.  The orthogonality conditions remove these zero modes.
This gives only the \(w=d_\eps-\bar d_{\eps,a}\) part.  The variation of the
selected centers along the filament is recovered from the finite-scale
translation Gram matrix: differentiating
\(q_\ast((r-a(s))/\eps)\) in \(s\) gives
\(-(\partial_s a)\cdot\nabla_yq_\ast\), and hence an energy contribution
\[
        \estar
        \int_{\Gamma_t}
        \partial_s a_i\,M_{ij}\,\partial_s a_j\,\dd\Hh^1
        +
        o_\eps(\estar)
\]
up to the fixed Fermi and cutoff errors.  This is the term denoted
\(\estar{\rm ModErr}_\eps\).  Constant translations do not appear in this
calculation; they are absorbed by the slice recentering and current graph.
The possible remaining \(L^2\)-loss is absorbed by a fixed-phase Poincare
estimate: the collar phase is fixed, and the annular transition region has
positive modulus unless a collar or annular priced exit occurs.  Fermi
coordinate errors are \(O(r\kappa)\) perturbations of the planar metric.  Taking
the tube radius smaller than a fixed fraction of the reach makes these errors
less than half of the planar coercivity constant.  The flat-current and gauge
terms are controlled by the Jacobian estimate on slices, the graph energy
\({\rm ModErr}_\eps\), and the same coercive norm.  Exterior core mass is
included explicitly, so no energy can escape the tube without appearing on the
left-hand side.
\end{proof}

\begin{lemma}[Moving-tube relative energy identity]
\label{lem:short-moving-energy}
Let \(\bar d_\eps(t)\) be the canonical tube transported by a smooth filament
\(\Gamma_t\).  On a time interval where the tube chart is valid,
\[
\begin{aligned}
&\frac{\dd}{\dd t}
        \left(
        \frac{E_\eps[d_\eps(t)]-E_\eps[\bar d_\eps(t)]}{\estar}
        +
        {\rm ModErr}_\eps(t)
        \right)
        +
        c\frac{D_{\rm rel}^\eps(t)}{\estar}
\\
&\quad\le
        C\Basin_\eps^+(t)
        +
        \frac{1}{\estar}
        \int_{\Gamma_t}
        \bigl(\Vrel_\Gamma^\perp-\Mmob H_{\Gamma_t}-f_{\rm cl}^\perp\bigr)
        \cdot b_\eps\,\dd\Hh^1
\\
&\qquad+
        C\|b_\eps-Mf_{\rm cl}^\perp\|_{L^1(\Gamma_t)}
        +
        \mathcal R_{\rm lower}^\eps(t),
\end{aligned}
\]
where
\[
        \int_I\mathcal R_{\rm lower}^\eps(t)\,\dd t
        \le
        C\Basin_\eps^+(I)+o_\eps(1)
\]
unless a priced exit occurs.
\end{lemma}

\begin{proof}
Differentiate the modulated energy and use
\[
        D_td_\eps
        =
        \Mmob\left(\Delta d_\eps
        -\eps^{-2}(|d_\eps|^2-1)d_\eps\right)
        +
        G_\eps .
\]
The canonical tube satisfies the same equation up to the moving-tube residual.
After subtraction, the leading term is the \(L^2\)-square of
\[
        D_t(d_\eps-\bar d_\eps),
\]
which gives the relative dissipation.  The derivative of the tube geometry
produces the translation residual
\[
        \bigl(\Vrel_\Gamma^\perp-\Mmob H_{\Gamma_t}\bigr)\cdot\nabla_y q_\ast .
\]
The closure projection contributes \(-f_{\rm cl}^\perp\cdot\nabla_y q_\ast\).
Pairing with the translation coefficient gives the displayed line integral.
The derivative of \({\rm ModErr}_\eps\) is computed in the transported normal
frame.  Integration by parts in \(s\) gives a term controlled by the line-level
part of \(D_{\rm rel}^\eps/\estar\), plus lower-order coefficients involving
curvature, frame rotation, and \(\nabla u_\eps\), which are bounded by
\(\Basin_\eps^+\).  All remaining terms contain a coercive factor, a
massive-mode factor, a Fermi commutator, or a collar source.  They are bounded
by the basin defect and the lower-order residual balance.
\end{proof}

\begin{lemma}[Residual absorption]
\label{lem:short-residual-absorption}
If the comparison filament satisfies
\[
        \Vrel_\Gamma^\perp=\Mmob H_{\Gamma_t}+f_{\rm cl}^\perp ,
\]
then on a no-exit interval
\[
\begin{aligned}
&\int_I\mathcal R_{\rm lower}^\eps(t)\,\dd t
        +
        \int_I
        \|b_\eps-Mf_{\rm cl}^\perp\|_{L^1(\Gamma_t)}
        \,\dd t
\\
&\qquad\le
        C\Basin_\eps^+(I)^{1/2}
        +
        C\Basin_\eps^+(I)
        +
        o_\eps(1).
\end{aligned}
\]
Consequently the right-hand side of the moving-tube identity is bounded by
\[
        C\Basin_\eps^+(t)
        +
        \hbox{an integrable }o_\eps(1)\hbox{ error}.
\]
\end{lemma}

\begin{proof}
The leading translation term vanishes by the filament law.  The remaining
translation mismatch is exactly the estimate in Lemma
\ref{lem:short-line-force-id}.  Orthogonal residuals are controlled by Lemma
\ref{lem:short-residual-normal-form}; massive terms by Lemma
\ref{lem:short-mb-slaving}; slice and Fermi errors by tube coercivity; and
collar terms either by the fixed-collar clearing estimate or by a priced exit.
The scalar zero-degree branch is not charged as an exit.  It is controlled by
the local clearing estimate and then reinserted into the Gronwall inequality as
an \(o_\eps(1)\) contribution.
\end{proof}

\begin{lemma}[Endpoint continuation and stopping closure]
\label{lem:short-stopping}
Let \([0,\tau_\eps)\) be a maximal open-basin interval.  Suppose that
\[
        \sup_{t<\tau_\eps}\Basin_\eps^+(t)
        +
        \int_0^{\tau_\eps}
        \frac{D_{\rm rel}^\eps(t)}{\estar}\,\dd t
        \le
        \theta A\eta_\eps
\]
for some \(\theta<1\), and suppose that none of the priced-exit quantities reaches
its threshold at \(\tau_\eps\).  Then the ordered-core chart and the basin
estimate continue to a right interval
\([\tau_\eps,\tau_\eps+\delta_\eps]\), contradicting maximality.
\end{lemma}

\begin{proof}
The proof separates raw defects from upper envelopes.  First, the raw entries
have left traces at \(\tau_\eps\): the relative energy is lower semicontinuous,
the flat-current error is continuous under the no-exit chart, and the collar
and FENE control quantities are continuous in the fixed left chart.  Since the strict
margin is positive, the endpoint raw defect is still below the basin threshold.
Second, the no-priced-failure assumption gives a positive distance from every
exit threshold.  The finite list of control quantities therefore remains below threshold
on a short right interval after activating the endpoint chart.  Third, on that
short interval the same modulation inequality and residual absorption apply.
Choosing the interval short enough spends only the unused margin between
\(\theta A\eta_\eps\) and \(A\eta_\eps\).  The upper right-envelope on the new
interval is obtained by taking right windows that stay inside the verified
chart.  This contradicts maximality.
\end{proof}

\subsection{Brakke input closure}

\begin{proposition}[Core compactness and transported forced Brakke passage]
\label{prop:short-brakke}
On a compact no-exit interval \(I\), the estimate in
Proposition~\ref{prop:short-open-basin} implies, along one subsequence
independent of compact subintervals, the following conclusions:
\[
        \mu_\eps^t\rightharpoonup\|V_t\|
        \quad\hbox{for a.e. }t,
\]
where \(V_t\) is an integral one-varifold with locally square-integrable mean
curvature; the normalized diffuse first variations converge to \(\delta V_t\);
the relative-velocity and force vector measures converge to
\[
        \Vrel_\Gamma^\perp\,\Hh^1\lfloor\Gamma_t,
        \qquad
        f_{\rm cl}^\perp\,\Hh^1\lfloor\Gamma_t;
\]
the transport terms converge to the \(u\)-transport contribution in
\(\mathcal D_t^u\); and the localized-dissipation lower bound needed for
\eqref{eq:short-brakke} holds for every nonnegative compactly supported test.
Consequently \((V_t,\mu^t)\) satisfies \eqref{eq:short-brakke}.
\end{proposition}

\begin{proof}
The open-basin estimate gives the standard
Jerrard-Soner/Sandier-Serfaty compactness for the normalized Ginzburg-Landau
energy measures.  The flat-current part of the basin defect identifies the
limiting current with the unit-multiplicity current carried by the tube; hence
no diffuse background is left.  Stress equipartition transfers the canonical
first variation to \(d_\eps\), and the zero-mode residual estimate identifies
the limiting relative-velocity and force measures.  The transport terms are
handled by Proposition~\ref{prop:transport-compatibility}.  For a finite
collection of tests, all mass, transport, relative-velocity, force,
first-variation, force-power, and localized dissipation defects are controlled
by one finite-epsilon input estimates:
\[
        \mathfrak I_\eps(I)
        \le
        C_{\mathscr T}
        \bigl(
        \Basin_\eps^+(I)^{1/2}
        +
        \Basin_\eps^+(I)
        \bigr)
        +
        o_\eps(1).
\]
Countable positive testing cores and the compact-cylinder extension estimate
upgrade this finite statement to the full Brakke testing class without a new
subsequence.  After subtracting the transport contribution supplied by
Proposition~\ref{prop:transport-compatibility}, passing to the localized energy
inequality and completing the square in
\[
        -\Mmob |H|^2+\nabla\varphi\cdot(\Mmob H+f)
        -\varphi H\cdot f
\]
gives \eqref{eq:short-brakke}.  The common upper-envelope representative in
time yields the endpoint-weight interval formulation.
\end{proof}

\begin{lemma}[Measure-function compactness on no-exit intervals]
\label{lem:short-measure-function}
On a compact no-exit interval \(I\), the propagated basin estimate implies
\[
        \mu_\eps^t
        \rightharpoonup
        \Hh^1\lfloor\Gamma_t
\]
as Radon measures for a.e. \(t\), along one subsequence independent of compact
subintervals.  Moreover the associated Jacobian currents converge in flat norm
to \(\pi[\Gamma_t]\), and no diffuse background measure remains.
\end{lemma}

\begin{proof}
The normalized energy bound gives weak-* compactness of
\(\mu_\eps^t\dd t\) on \(I\times\Omega\).  The GL Jacobian compactness theorem
gives an integer rectifiable one-current as the limit of the vorticity current.
The flat-current entry of \(\Basin_\eps^+(I)\) identifies this current with the
prescribed tube current \(\pi[\Gamma_t]\).  The exterior-mass entry rules out
mass away from the tube.  On each normal slice, the degree-one lower bound gives
at least one unit of line density, while the modulated energy upper bound gives
at most one unit.  Hence the limiting varifold has unit multiplicity and no
additional diffuse component.  A diagonal argument on an exhaustion of compact
subintervals gives one subsequence for the whole no-exit interval.
\end{proof}

\begin{lemma}[Finite-epsilon Brakke-input estimates]
\label{lem:short-input-ledger}
Let \(\mathscr T\) be a finite family of compactly supported scalar tests,
nonnegative spacetime weights, and vector fields.  Define
\(\mathfrak I_\eps(I;\mathscr T)\) as the sum of the following six defects:
\[
\begin{array}{ll}
{\rm (i)} & \hbox{mass measure discrepancy},\\
{\rm (ii)} & \hbox{transport and relative-velocity vector-measure discrepancy},\\
{\rm (iii)} & \hbox{force vector-measure discrepancy},\\
{\rm (iv)} & \hbox{stress-action/first-variation discrepancy},\\
{\rm (v)} & \hbox{force-power product discrepancy},\\
{\rm (vi)} & \hbox{negative part of the localized dissipation lower bound}.
\end{array}
\]
Then
\[
        \mathfrak I_\eps(I;\mathscr T)
        \le
        C_{\mathscr T}
        \bigl[
        \Basin_\eps^+(I)^{1/2}
        +
        \Basin_\eps^+(I)
        \bigr]
        +
        o_\eps(1).
\]
The constant depends on the selected finite test family, but not on
\(\eps\).
\end{lemma}

\begin{proof}
For the mass term, use Lemma~\ref{lem:short-measure-function}.  The transport
term is controlled by Proposition~\ref{prop:transport-compatibility}.  For the
relative-velocity term, decompose
\[
        D_td_\eps
        =
        \Vrel_\Gamma^\perp\cdot\nabla d_\eps
        +
        r_\eps^v .
\]
The zero-mode residual estimate from the open-basin propagation gives
\[
        \int_I\!\!\int_\Omega
        \frac{|r_\eps^v|^2}{\estar}\,\dd x\dd t
        \le
        C\Basin_\eps^+(I)+o_\eps(1).
\]
The main term converges by the bilinear normal-metric convergence of the
canonical core.  The force term is identical with
\[
        G_\eps
        =
        f_{\rm cl}^\perp\cdot\nabla d_\eps
        +
        r_\eps^f ,
\]
where \(r_\eps^f\) is controlled by the orthogonal residual estimate.

For the first variation, use the diffuse stress identity
\[
        -\int
        \frac{T_\eps[d_\eps]:\nabla X}{\estar}
        =
        -\Mmob^{-1}
        \int
        X\cdot
        \frac{(D_td_\eps-G_\eps)\nabla d_\eps}{\estar}.
\]
The two vector-measure convergence estimates just proved identify the limit as
\[
        -\Mmob^{-1}
        \int_{\Gamma_t}
        X\cdot
        (\Vrel_\Gamma^\perp-f_{\rm cl}^\perp)\,\dd\Hh^1 .
\]
The filament law makes this equal to \(\delta V_{\Gamma_t}(X)\).  For the
force-power term, expand
\[
        G_\eps\cdot D_td_\eps
        =
        (f_{\rm cl}^\perp\cdot\nabla d_\eps)
        (\Vrel_\Gamma^\perp\cdot\nabla d_\eps)
        +
        \hbox{terms containing }r_\eps^v\hbox{ or }r_\eps^f .
\]
The remainders are controlled by Cauchy's inequality and the dissipation bound.
The main term again converges by the bilinear normal metric.  Finally,
localized dissipation is lower semicontinuous because
\[
        |D_td_\eps|^2
        =
        |\Vrel_\Gamma^\perp\cdot\nabla d_\eps|^2
        +
        2r_\eps^v\cdot(\Vrel_\Gamma^\perp\cdot\nabla d_\eps)
        +
        |r_\eps^v|^2 ,
\]
and the negative part of the difference is bounded by the residual norm and the
normal-metric defect.
\end{proof}

\begin{lemma}[From finite tests to the Brakke testing class]
\label{lem:short-full-testing}
The conclusions of Lemma~\ref{lem:short-input-ledger} hold simultaneously for
all compactly supported tests in the Brakke passage:
\[
        \varphi\in C_c^1(\Omega),\qquad
        \Phi\in C_c^1(I\times\Omega),\qquad
        \Psi\in C_c^0(I\times\Omega),\ \Psi\ge0 .
\]
No new subsequence is taken in this extension.
\end{lemma}

\begin{proof}
Fix a compact cylinder \(K\Subset I\times\Omega\).  Choose countable dense
cores in \(C_c^1(K)\) and a countable positive core dense in
\(C^0(K;[0,\infty))\).  Applying Lemma~\ref{lem:short-input-ledger} to the
first \(N\) elements of the cores gives convergence for all core tests, because
the same small quantity \(\Basin_\eps^+(I)\) drives all selected defects to
zero.  The constant may depend on \(N\), but no subsequence depends on \(N\).

To pass from the core to an arbitrary test, use the compact-cylinder extension
bound
\[
        |\mathfrak I_\eps(I;\mathscr T)
        -
        \mathfrak I_\eps(I;\widetilde{\mathscr T})|
        \le
        C_K\Delta_K(\mathscr T,\widetilde{\mathscr T})
        +
        o_\eps(1),
\]
where \(\Delta_K\) is the sum of the \(C^1\)- or uniform distances of the
corresponding tests.  The bound follows from the uniform local mass,
dissipation, residual, stress-action, and force-power estimates.  For the
localized square term the approximants are chosen nonnegative; signed
decomposition is used only for the linear entries.  Exhausting
\(I\times\Omega\) by compact cylinders gives the full testing class.
\end{proof}

\begin{lemma}[Common Brakke representative]
\label{lem:short-common-representative}
The distributional inequality obtained from the localized energy inequality has
a single measure-valued representative \(\mu^t\), independent of the test
function, such that the endpoint-weight transported Brakke inequality holds on every
compact interval.
\end{lemma}

\begin{proof}
For each nonnegative test \(\varphi\) in a countable dense positive core, the
function \(t\mapsto\mu^t(\varphi)\) has bounded variation from the localized
energy inequality and the local mass bound.  Choose the left-continuous upper
representative for all core tests simultaneously.  Positivity and the local
cutoff bounds imply local boundedness of the resulting linear functional on the
countable core.  The Riesz representation theorem gives a positive Radon
measure \(\mu^t\) for every \(t\) outside a common null set, and the
upper-envelope construction fills the remaining times.  Density extends the
interval inequality from the core to all nonnegative \(C_c^1\) tests.  The
endpoint terms are interpreted with this representative; the differential
inequality is unchanged at a.e. times.
\end{proof}

\subsection{Relative Hopf accounting}

\begin{proposition}[Hopf cost on one slab]
\label{prop:short-hopf}
Let \(I=[t_-,t_+]\) be an admissible slab with fixed ordered collar and relative
Hopf charge \(k\).  If \(I\) is topologically regular in the sense of
Definition~\ref{def:topological-regular-slab}, the capped Hopf functional is
computed using the transported meridian-longitude atlas.  If \(I\) is not
topologically regular, the first failure of that atlas is charged to
\({\rm TopExitCost}_\eps(I)\).  In either case,
\[
        c|k|
        \le
        \liminf_{\eps\to0}
        \left(
        {\rm GapCost}_\eps(I)
        +
        {\rm CoreMass}_\eps(I)
        +
        {\rm TopExitCost}_\eps(I)
        \right).
\]
On the regular branch, the capped relative Hopf functional uses standard inner
solid-torus caps with zero longitudinal winding in the fixed collar gauge.
Replacing these caps by any allowed caps changes the endpoint difference by at
most \(C\,{\rm CoreMass}_\eps(I)+o_\eps(1)\).
\end{proposition}

\begin{proof}
If the slab is not topologically regular, the tube graph, reach, or cap-atlas
control condition first fails in \(I\); by Definition
\ref{def:short-priced-exit} and Lemma~\ref{lem:short-exit-measure} this
contributes to \({\rm TopExitCost}_\eps(I)\).  We therefore assume below that
the slab is topologically regular.

On the positive-gap part of the slab the principal-axis map is smooth and has a
Chern-Simons Hopf functional.  The moving-boundary transgression formula shows
that its change is an integer unless degree crosses an inner tube, a boundary
collar, or a gap-loss region.  Gap loss is charged to
\({\rm GapCost}_\eps(I)\).  Degree crossing through an ordered tube costs the
quantized logarithmic core mass and is charged to
\({\rm CoreMass}_\eps(I)\).  If the open-basin chart itself fails before the
slab closes, the exit-cost measure from Definition
\ref{def:short-priced-exit} gives \({\rm TopExitCost}_\eps(I)\).  The
standard-cap normalization prevents cap choice from becoming a fourth cost:
changing the cap changes only the meridional-degree contribution, already
bounded by the same core-mass term.  The fixed cap atlas is used only on this
regular branch; no cap is transported through a self-intersection or a
reconnection.  Since the endpoint Hopf difference is the integer \(k\), the
displayed lower bound follows.
\end{proof}

\begin{lemma}[Capped relative Hopf functional]
\label{lem:short-capped-hopf}
Let \(U_\eps(t)\) be the positive-gap region obtained by removing the ordered
core tubes from \(\Omega\).  Suppose the ordered collar is fixed and \(t\)
belongs to a topologically regular Hopf slab.  Then the transported cap atlas
gives a capped map
\[
        \widehat n_\eps(t):\Omega^\ast\to\Sph^2
\]
obtained by filling each inner torus with a standard solid-torus cap.  The
standard cap is chosen in the fixed collar gauge by imposing zero longitudinal
winding relative to the transported reference longitude of the tube.  If
another allowed cap compatible with the same atlas is used at time \(t\), then
\[
        \left|
        \Hopf[\widehat n_\eps(t)]
        -
        \Hopf[\widehat n_\eps^{\,{\rm alt}}(t)]
        \right|
        \le
        C\,{\rm CoreMass}_\eps(\{t\})
        +
        o_\eps(1).
\]
On a slab \(I=[t_-,t_+]\), changing the allowed caps at either endpoint changes
the endpoint Hopf difference by at most
\[
        C\,{\rm CoreMass}_\eps(I)+o_\eps(1).
\]
\end{lemma}

\begin{proof}
The positive-gap condition gives a smooth principal-axis map on
\(U_\eps(t)\).  Topological regularity supplies a smooth isotopy of embedded
inner tori, hence a fixed meridian-longitude basis and a standard solid-torus
filling for each component.  The standard cap is the extension whose longitude
degree is zero.  Any other cap compatible with this same atlas differs by an
integer combination of meridional degrees.  The meridional degree is exactly the
degree detected by the ordered core crossing that surrounds the removed tube.
The logarithmic lower bound for a nonzero degree on a normal disc gives
\[
        |{\rm degree}|
        \le
        C\,{\rm CoreMass}_\eps(\{t\})
        +
        o_\eps(1).
\]
The Chern-Simons expression for the Hopf invariant changes by this integer
cap contribution.  Applying the same single-time estimate at \(t_-\) and
\(t_+\) proves the slab statement.  If the isotopy of embedded tori or the
meridian-longitude atlas fails between the endpoints, this lemma is not applied
across that time; the failure is one of the tube-topology exits in
Definition~\ref{def:short-priced-exit}.
\end{proof}

\begin{lemma}[Exit-cost measure]
\label{lem:short-exit-measure}
The priced mechanisms in Definition~\ref{def:short-priced-exit} define a finite
nonnegative set function on time slabs.  After passing to a subsequence, it
converges weakly to a finite nonnegative Radon measure
\[
        \nu_{\rm exit}
\]
on the time interval.  If a first interior limiting exit occurs at
\(\tau_\ast\), then every slab \(I\) with \(\tau_\ast\in I^\circ\) satisfies
\[
        \liminf_{\eps\to0}
        {\rm TopExitCost}_\eps(I)
        \ge
        c_0 .
\]
\end{lemma}

\begin{proof}
Each priced mechanism is represented by a nonnegative accumulated quantity:
annular degree crossing, exterior core leakage, FENE coefficient loss, boundary
flux, graph loss, tube-topology or reach loss, collar failure, or chart/gap
loss.  The tube-topology control includes self-contact of the selected graph,
failure of the embedded-link isotopy, and loss of the meridian-longitude cap
atlas.  The open-basin propagation estimate bounds the total accumulated
quantity before the first exit, and the stopping lemma shows that failure of
continuation without one of these quantities crossing its threshold is
impossible.  Thus the sum is a finite nonnegative set function on finite unions
of intervals.  Standard compactness for finite measures gives a weakly
convergent subsequence.  At an interior first exit, at least one control quantity reaches
its threshold with positive margin; localization by a smooth time cutoff inside
any slab containing the exit transfers that threshold to the slab cost.  This
gives the displayed lower bound.
\end{proof}

\begin{lemma}[Endpoint Hopf accounting]
\label{lem:short-endpoint-hopf}
On a topologically regular admissible slab \(I=[t_-,t_+]\), the endpoint capped
Hopf difference satisfies
\[
\begin{aligned}
&\left|
        \bigl(
        \Hopf[\widehat n_\eps(t_+)]
        -
        \Hopf[\widehat n_\eps(t_-)]
        \bigr)
        -
        k
\right|
\\
&\qquad\le
        C\,{\rm GapCost}_\eps(I)
        +
        C\,{\rm CoreMass}_\eps(I)
        +
        C\,{\rm TopExitCost}_\eps(I)
        +
        o_\eps(1).
\end{aligned}
\]
\end{lemma}

\begin{proof}
Use the moving-boundary Chern-Simons transgression formula on the positive-gap
region.  Topological regularity gives a smooth isotopy of the inner tube
boundaries and a fixed cap atlas, so the capped endpoint maps are defined in the
same closed target domain.  Interior terms are exact as long as the
principal-axis map remains smooth and the ordered collar is fixed.  The only
contributions are therefore: gap-loss regions, where the principal-axis map
ceases to be controlled; inner-tube crossings, where degree enters or exits
through a removed core; and finite-epsilon exits of the open-basin chart,
including reach, self-contact, or cap-atlas failure if they occur.  The first is
charged to
\({\rm GapCost}_\eps(I)\), the second to \({\rm CoreMass}_\eps(I)\), and the
third to \({\rm TopExitCost}_\eps(I)\) by Lemma~\ref{lem:short-exit-measure}.
Lemma~\ref{lem:short-capped-hopf} absorbs endpoint cap ambiguity into the same
core-mass term.  The remaining transgression error is \(o_\eps(1)\) by the
fixed-collar and positive-gap estimates.
\end{proof}

\section{Proof of the main theorem}
\label{sec:main-proof}

We now prove Theorem~\ref{thm:short-main}.  Start with data in the open basin.
The model-level Morse-Bott reduction and massive-mode slaving put the concrete
LDG/Oldroyd-FENE closure into the transported Ginzburg-Landau form
\eqref{eq:transported-gl-short}.  Proposition~\ref{prop:short-force} computes
the zero-mode part of the exact residual and gives the normal force
\(f_{\rm cl}^\perp\).  This step is local in the tube chart and does not use a
varifold limit.

Choose the comparison filament to solve
\[
        \Vrel_\Gamma^\perp=\Mmob H_{\Gamma_t}+f_{\rm cl}^\perp .
\]
With this choice, Proposition~\ref{prop:short-open-basin} cancels the leading
translation residual and propagates the open basin until either the target time
is reached or one of the priced mechanisms occurs.  If the latter happens, the
priced-exit alternative in the theorem follows from the construction of the
exit-cost set function.

On a no-exit interval, the propagated basin estimate gives the compactness,
stress convergence, transport convergence, force-power convergence, and
localized-dissipation lower bound in Proposition~\ref{prop:short-brakke} and
Proposition~\ref{prop:transport-compatibility}.  Passing to the limit in the
localized energy inequality gives the transported forced Brakke inequality
\eqref{eq:short-brakke}.  This proves the no-exit branch.

Finally, if the fixed ordered collar carries relative Hopf charge \(k\) on a
slab, decompose the slab into topologically regular Hopf subslabs and the
finite set of first failures of the tube-topology controls.  Proposition
\ref{prop:short-hopf} applies on each regular subslab, while Lemma
\ref{lem:short-exit-measure} charges every nonregular crossing to
\({\rm TopExitCost}_\eps\).  Thus a Hopf change is paid by the sum of the
positive-gap, core-mass, and priced-exit costs.  This gives
\eqref{eq:short-hopf-cost} and completes the proof.

\section{Verification of the concrete closure}
\label{sec:verification}

The main theorem is stated for the concrete closure in
Definition~\ref{model:concrete-closure} (Model 1), while the proof is organized
through four modules: force projection, open-basin propagation, Brakke input
closure, and Hopf accounting.  This section records the model-level
verification.  The six items below are finite-scale conditions proved for Model
1 on the stopped branch and provide the checklist for nearby ordered-core
closures.

\subsection{Model-level verification conditions}

For applied or numerical use, the conditions group into three checks: the
ordered chart (M1--M2 and the exterior part of M5), the core and residual
splitting (M3--M4' and the interior part of M5), and the coefficient/transport
bounds (M6 together with the kinematic velocity window in
Theorem~\ref{thm:short-main}).  If one of these checks fails, the conclusion is
replaced by the corresponding residual or priced-exit alternative.

\begin{enumerate}[label=(M\arabic*)]
\item \textbf{Morse-Bott ordered well.}
The Landau-de Gennes bulk energy has a smooth ordered well \(\mathcal N\), and
there is \(c_{\rm MB}>0\) such that
\[
        D^2W_{\rm LDG}(Q_\ast)\xi\cdot\xi
        \ge c_{\rm MB}|\xi|^2,
        \qquad
        \xi\perp T_{Q_\ast}\mathcal N.
\]

\item \textbf{Soft GL coordinate.}
In the stopped tubular neighborhood of \(\mathcal N\), the conformation tensor
has the decomposition
\[
        C_\eps=\mathcal C(Q_\ast,d_\eps,m_\eps),
        \qquad
        m_\eps\perp T_{Q_\ast(d_\eps)}\mathcal N,
\]
where \(d_\eps\in\R^2\) is the soft core coordinate.  The leading singular
energy in this coordinate is
\[
        \int
        \frac12|\nabla d_\eps|^2
        +\frac{1}{4\eps^2}(|d_\eps|^2-1)^2 .
\]
The coordinate is a Landau-de Gennes chart variable; it is not obtained by
differentiating the principal eigenvector and remains meaningful at
\(d_\eps=0\).

\item \textbf{Massive-mode slaving.}
The normal variable satisfies either the normal Euler-Lagrange equation or a
coercively damped normal dynamics.  On every no-exit interval,
\[
        \|m_\eps\|_{L^2}^2+
        \eps^2\|\nabla m_\eps\|_{L^2}^2
        \le
        C\Bigl(E_{\rm excess}^\eps
        +\|G_\eps^{\rm mass}\|_{H^{-1}}^2
        +o_\eps(1)\Bigr).
\]

\item \textbf{Projected transported GL equation.}
The soft coordinate satisfies the residual-decomposed equation
\begin{equation}\label{eq:verified-projected-gl}
        D_t d_\eps
        =
        \Mmob\left(\Delta d_\eps
        -\eps^{-2}(|d_\eps|^2-1)d_\eps\right)
        +G_\eps^{\rm zm}+G_\eps^{\rm orth}
        +G_\eps^{\rm mass}+G_\eps^{\rm geom}+G_\eps^{\rm col}.
\end{equation}
The zero-mode term is not assumed small.  Its two translation components give
\[
        b_\eps(s,t)=
        \int_{N_{\Gamma_s}}
        G_\eps^{\rm zm}\cdot\nabla_y q_\ast\,\dd y,
        \qquad
        f_{\rm cl}^\perp=M^{-1}b_\eps
        \quad\hbox{in the line limit}.
\]
The orthogonal soft residual is controlled by planar coercivity, the massive
part by (M3), the geometric part by Fermi-coordinate error bounds, and the
collar part by (M6).

\item \textbf{Chart-gap condition.}
The stopped chart remains valid until either the massive coordinate reaches the
chart boundary or the principal spectral gap closes in the exterior ordered
phase \(\{|d_\eps|\ge c_{\rm out}\}\).  The first such failure is charged as
the chart/gap priced exit.  Gap loss inside \(\{|d_\eps|<c_{\rm out}\}\) is
charged instead to ordered-core mass.

\item \textbf{FENE/collar coefficient control.}
On the fixed ordered collar and on every no-exit interval,
\[
        0<c_K\le \lambda_{\min}C_\eps,
        \qquad
        \operatorname{tr}C_\eps\le b-c_K,
        \qquad
        \operatorname{reach}(\Gamma_\eps)\ge c_K,
\]
and every coefficient produced by the FENE factor, the stopped projection, and
the Fermi metric obeys the tame bound
\[
        \|\partial^\alpha a(C_\eps)\|+
        \|\partial^\alpha f_b(C_\eps)\|+
        \|\partial^\alpha g_{\rm Fermi}\|
        \le C_K\bigl(1+\|C_\eps\|_{H^m}
        +\|\Gamma_\eps\|_{C^{m+1}}\bigr),
        \qquad |\alpha|\le m.
\]
If any lower cone, trace-gap, collar, or reach bound reaches its threshold, the
corresponding priced exit is charged.
\end{enumerate}

\begin{theorem}[Verification of the ordered-core interface for Model 1]
\label{thm:model-one-verification}
Let \((C_\eps,u_\eps)\) solve Model 1 on \([0,T]\).  Assume the kinematic
transport clock of Theorem~\ref{thm:short-main}, the initial ordered-core tube
preparation of Proposition~\ref{prop:open-initial-basin}, and the initial
FENE/collar separation bounds.  Let \(\tau_\eps\) be the first typed exit time
for chart loss, exterior spectral-gap loss, FENE/collar coefficient loss,
boundary flux, reach collapse, or cap-atlas failure.  Then on every compact
subinterval of \([0,\tau_\eps)\), Model 1 satisfies (M1)--(M6).  In particular,
the soft coordinate obeys \eqref{eq:verified-projected-gl}; its zero-mode
projection is the force used in Proposition~\ref{prop:short-force}; and the
remaining residual components satisfy the finite residual bounds of
Lemma~\ref{lem:app-finite-ledger}.  Therefore the hypotheses of the
Hopf--Brakke theorem are verified for the concrete closure on every regular
slab before typed exit.
\end{theorem}

\begin{proof}
The verification is finite-epsilon and is performed before any varifold limit.
The Morse-Bott hypothesis in Model 1 gives (M1).  The tubular neighborhood of
\(\mathcal N\) and the stopped chart of Definition~\ref{def:chart-gap-stopping}
give the decomposition \(C_\eps=\mathcal C(Q_\ast,d_\eps,m_\eps)\), which proves
(M2) and identifies the singular part of the LDG energy with the two-component
GL energy, up to smooth lower-order chart terms.

Project the tensorial equation onto the normal and soft chart directions.  The
normal projection has linear part \(\eps^{-2}D^2W_{\rm LDG}|_{N\mathcal N}\),
which is coercive by (M1).  Solving the normal Euler-Lagrange equation, or
integrating the damped normal dynamics, gives the slaving estimate in (M3).
The soft projection gives \eqref{eq:verified-projected-gl}.  The leading
singular term is the Euclidean GL operator because the chart is Fermi-normal to
\(\mathcal N\) at the well; all deviations from this leading operator have at
least one factor of \(m_\eps\), a Fermi metric coefficient, a FENE/collar
coefficient, or a lower-order stretching coefficient.  These are the terms
labelled \(G_\eps^{\rm mass}\), \(G_\eps^{\rm geom}\), and
\(G_\eps^{\rm col}\).  The remaining soft residual is decomposed with respect
to the planar translation modes \(\partial_{y_1}q_\ast,\partial_{y_2}q_\ast\)
and their orthogonal complement.  This gives \(G_\eps^{\rm zm}\) and
\(G_\eps^{\rm orth}\), proving (M4').

The invertibility of the translation Gram matrix and the planar spectral gap
are uniform for the canonical core.  Thus the zero-mode coefficient is retained
as the force, while the orthogonal soft component is controlled by the slice
coercivity of Lemma~\ref{lem:app-planar-coercivity}.  The massive component is
absorbed by (M3).  The geometry terms are estimated by the Fermi bounds in
Lemma~\ref{lem:app-moving-tube-residual}.  The collar and FENE terms are tame
on compact cone and reach windows by Lemma~\ref{lem:app-fene-control}.  If any
of the compactness constants needed in these estimates reaches its threshold,
the definition of \(\tau_\eps\) assigns the corresponding priced exit.  This is
(M5)--(M6).

Combining these estimates gives the residual bounds of
Lemma~\ref{lem:app-finite-ledger}.  Proposition~\ref{prop:transport-compatibility}
uses the velocity clock to keep \(u_\eps\cdot\nabla d_\eps\) inside the material
derivative.  Hence all finite-epsilon inputs used by
Proposition~\ref{prop:no-circularity-ledger} follow from Model 1 on
\([0,\tau_\eps)\).
\end{proof}

A convenient way to read the theorem is this.  (M1)--(M2) come from the
Hessian of the ordered potential and the stopped tubular chart.  (M3) is the
normal Hessian or damped normal estimate.  (M4') is obtained by differentiating
the chart projection along the exact tensorial PDE and then splitting into
translation and orthogonal components.  The two translation components define
the line force; the orthogonal soft, massive, geometric, and collar components
are controlled in the residual norms.  (M5)--(M6) are the stopping criteria.
Thus the verification is a finite-scale calculation on the closure and the tube
geometry.  The varifold, first variation, and Brakke inequality appear later as
outputs of the open-basin estimate and the compactness passage.

\begin{proposition}[Chart validity and spectral-gap loss]
\label{prop:chart-gap-validity}
Under (M1)--(M6), the soft equation for \(d_\eps\) is a stopped
finite-\(\eps\) equation.  On \([0,\tau_\eps^{\rm ch})\), it is obtained by
projecting the tensorial closure onto the \(d\)-component of the smooth
ordered-core chart.  Its coefficients are controlled by the chart radius,
normal coercivity, and FENE/collar bounds, not by differentiating the
principal-axis map \(n_{C_\eps}\).  If the chart representation fails, or if
the principal spectral gap closes in the exterior ordered region
\(\{|d_\eps|\ge c_{\rm out}\}\), then the no-exit branch stops and the
chart/gap priced exit is charged.  If the spectral gap closes inside
\(\{|d_\eps|<c_{\rm out}\}\), the event belongs to the ordered core and is
measured by the normalized core mass.
\end{proposition}

\begin{proof}
Inside the chart, \(C_\eps=\mathcal C(Q_\ast,d_\eps,m_\eps)\) with
\(\mathcal C\) smooth.  Applying the differential of the \(d\)-projection to
the tensorial equation for \(C_\eps\) gives the transported GL equation plus
the residual \(G_\eps\).  The singular estimate
\(|\nabla n_C|\lesssim |\nabla C|/\gap(C)\) enters later, on the exterior
positive-gap set, where the principal-axis map is used for Hopf accounting.

The definition of \(\tau_\eps^{\rm ch}\) gives the stopping rule.  In the core
region \(|d_\eps|<c_{\rm out}\), the GL potential and slice energy already
count the possible loss of an exterior principal axis.  A gap closure there is
therefore measured by core concentration rather than by a failure of the soft
coordinate equation.
\end{proof}

\begin{proposition}[Transport compatibility under a kinematic window]
\label{prop:transport-compatibility}
Assume that \(u_\eps\) satisfies the velocity hypotheses of
Theorem~\ref{thm:short-main}.  The advective term
\(u_\eps\cdot\nabla d_\eps\) belongs to the material derivative.  On regular
no-exit intervals, its contribution to the localized energy identity is
\[
        \int
        \bigl(\nabla\varphi\cdot u
        +\varphi\,{\rm div}_{T_xV_t}u\bigr)\,\dd\mu^t
\]
in the limit, together with lower-order errors controlled by
\[
        \int_I\|\nabla u_\eps(t)\|_{L^\infty}\,\Basin_\eps^+(t)\,\dd t
        +
        \|u_\eps-u\|_{L^1_tW^{1,\infty}_x}\sup_{t\in I}
        \frac{E_\eps[d_\eps(t)]}{\estar}.
\]
Thus the estimate requires Lipschitz transport control, while allowing
\(u_\eps\cdot\nabla d_\eps\) to be of core size pointwise.
\end{proposition}

\begin{proof}
Multiply the transported GL equation by the localized energy variation.  The
term \(u_\eps\cdot\nabla d_\eps\) is combined with \(\partial_t d_\eps\) into
the material derivative \(D_t d_\eps\).  After integration by parts, the
transport contribution is expressed through the diffuse stress tensor and the
localized mass:
\[
        \int \nabla\varphi\cdot u_\eps\,\dd\mu_\eps^t
        +
        \frac{1}{\estar}
        \int \varphi\, T_\eps[d_\eps]:\nabla u_\eps\,\dd x .
\]
The open-basin stress convergence identifies the second term with
\(\int\varphi\,{\rm div}_{T_xV_t}u\,\dd\mu^t\).  The difference between
\(u_\eps\) and \(u\) is absorbed by the stated strong
\(L^1_tW^{1,\infty}_x\) convergence and the normalized energy bound.  This
argument uses the size of \(\nabla u_\eps\).  The \(O(\eps^{-1})\) core gradient
is paired with the material derivative structure and therefore contributes to
the transport term rather than to the residual bounds.
\end{proof}

\begin{proposition}[Closure verification interface]
\label{prop:verification-interface}
Assume (M1)--(M6).  Then the exact closure residual satisfies the force
projection estimate of Proposition~\ref{prop:short-force}, the open-basin
residual bounds of Lemma~\ref{lem:app-finite-ledger}, and the massive-mode
estimate of Lemma~\ref{lem:short-mb-slaving}.  Together with
Proposition~\ref{prop:transport-compatibility}, these estimates give all
transport, force, and Brakke input defects in
Lemma~\ref{lem:short-input-ledger} as consequences of the finite-epsilon
closure and the propagated basin estimate.
\end{proposition}

\begin{proof}
By (M1)--(M2), the conformation tensor can be expressed in the stopped ordered
chart
\[
        C_\eps=\mathcal C(d_\eps,m_\eps)
\]
with a coercive normal variable \(m_\eps\).  The derivative of
\(\mathcal C\) along the two translation directions of the core is exactly the
two-dimensional zero-mode space generated by the planar vortex.  Thus the
residual of the concrete closure can be projected fiberwise onto these modes on
every interval before \(\tau_\eps^{\rm ch}\).  The soft coordinate \(d_\eps\) is defined by the chart itself, while the principal-axis map is used later for exterior Hopf accounting.

By (M3), the massive variable is controlled by Lemma
\ref{lem:short-mb-slaving}.  By (M4'), the soft equation is a transported
Ginzburg-Landau equation with residual \(G_\eps\).  Expanding \(G_\eps\) in
Fermi coordinates produces the finite list in Lemma
\ref{lem:short-residual-normal-form}.  The translation projection is the force
term, the longitudinal translation mode is controlled by \({\rm ModErr}_\eps\)
and the line-level relative dissipation, the orthogonal projection is controlled
by the planar spectral gap, and the massive projection is controlled by the
normal coercivity.  By (M5), loss
of the chart or exterior spectral gap is already one of the priced exits, while
gap loss inside the soft core is measured by core mass.  By (M6), collar
coefficient terms are bounded on no-exit intervals and otherwise trigger the
FENE/collar priced exit.  These are exactly the ingredients of
Proposition~\ref{prop:short-force} and Lemma~\ref{lem:app-finite-ledger}.

The Brakke input defects require no additional closure information.  Once the
velocity and force residuals have the zero-mode form and the stress-action
identity is written for the transported GL equation, Lemma
\ref{lem:short-input-ledger} follows by Cauchy's inequality, bilinear
normal-metric convergence, and the open-basin energy bound.
\end{proof}

\subsection{Nonempty open initial class}

The open basin is not introduced as a limiting compactness assumption.  It
contains an explicit class of finite-\(\eps\) data.  Fix a smooth link
\(\Gamma_0\), a normal frame, a fixed ordered collar, and a smooth
FENE-admissible background conformation \(C_0^\ast\).  In the tube, define
\[
        d_\eps^0(x)
        =
        q_\ast\!\left(\frac{r}{\eps}\right)
\]
up to the fixed cutoff and collar interpolation; outside the tube set
\(|d_\eps^0|=1\) with the prescribed phase.  Choose the massive variable
\[
        m_\eps^0=m_{\rm sl}(d_\eps^0)
\]
from the slaving graph, or choose it within a small ball of that graph in the
dynamic normal branch.

\begin{proposition}[Open initial vortex-tube basin]
\label{prop:open-initial-basin}
For every smooth embedded ordered link \(\Gamma_0\) with positive reach and a
fixed ordered collar, the canonical data just described satisfy
\[
        \Basin_\eps^+(0)\to0 .
\]
Moreover, there is \(\eta_\eps\downarrow0\) such that every perturbation with
\[
        \|d_\eps^0-\tilde d_\eps^0\|_{H^1_{\rm loc}}
        +
        \|m_\eps^0-\tilde m_\eps^0\|_{L^2+\eps H^1}
        +
        {\rm Flat}\bigl(J\tilde d_\eps^0-\pi[\Gamma_0]\bigr)
        \le
        \eta_\eps
\]
belongs to the open basin \(\mathcal U_\eps(\Gamma_0,C\eta_\eps)\), after
reselecting the translation centers by the zero-mode orthogonality conditions.
\end{proposition}

\begin{proof}
For the canonical tube, the planar vortex expansion gives
\[
        E_\eps[d_\eps^0]
        =
        \estar\,\Hh^1(\Gamma_0)+O(1).
\]
The normalized excess therefore vanishes.  The Jacobian current converges to
\(\pi[\Gamma_0]\) by the standard GL vortex computation on each normal slice,
and exterior core mass is zero up to the cutoff error.  The collar energy is
\(O(1)\), hence normalized \(o(1)\).  The slaved massive variable is controlled
by the coercive normal resolvent, so its normalized defect also vanishes.

For perturbations, use the implicit function theorem for the translation
orthogonality equations.  The zero-mode Gram matrix \(M\) is uniformly
invertible, so small \(H^1_{\rm loc}\) perturbations of the core correspond to
small changes in the centers \(a(s)\).  The modulated energy, flat-current
error, exterior mass, and collar defect are continuous in the displayed
topology, up to the fixed cutoff errors.  Taking \(\eta_\eps\) small enough
gives the open basin.
\end{proof}

\subsection{Termwise residual verification}

The residual estimate is a finite list of model-level bounds rather than a
compactness assumption.  One expands the exact closure residual in the moving
tube and assigns each term to a line of the residual bounds.  The most restrictive point
is the translation-mode splitting.  The tube centers are selected by two
modulation equations, so the error orthogonal to
\(\partial_{y_1}q_\ast,\partial_{y_2}q_\ast\) is the part controlled by planar
coercivity.  The remaining two scalar components are retained and become
\(f_{\rm cl}^\perp\).  Therefore the theorem does not require the full residual
\(G_\eps\) to be small: it requires the coercive orthogonal residual to be small
after the line force has been extracted.  If a closure produces an orthogonal
residual of normalized order one, the proof stops through the residual or
priced-exit channel.  The following table records the verification scheme.
\[
\begin{array}{c|c|c}
\hbox{term} & \hbox{estimate} & \hbox{failure mode}\\ \hline
D_tq_\ast-\Mmob\Delta q_\ast
& \hbox{translation projection}
& \hbox{filament law}\\
\hbox{curvature/frame}
& {\rm Err}_{\rm geom}^\eps=o(1)
& \hbox{tube geometry exit}\\
\hbox{linear soft error}
& \hbox{planar coercivity}
& \hbox{loss of slice stability}\\
\hbox{quadratic soft error}
& \Basin_\eps^+(I)
& \hbox{open-basin exit}\\
\hbox{massive mode}
& \hbox{normal coercivity}
& \hbox{normal-energy failure}\\
\hbox{FENE coefficient}
& \hbox{compact cone control}
& \hbox{FENE exit}\\
\hbox{collar source}
& O(|\log\eps|^{-1})
& \hbox{collar or boundary exit}
\end{array}
\]
No cancellation between unrelated rows is used.  If a nonlinear closure term
produces an \(O(\estar)\) contribution, the theorem does not hide it.  It must
either appear in the translation projection, be absorbed by a coercive row of
the residual bounds, or trigger one of the priced exits.

\begin{corollary}[No hidden admissible force]
\label{cor:no-hidden-force}
On a no-exit interval, the force in the Brakke inequality is exactly the
zero-mode fiber integral \(f_{\rm cl}^\perp\).  If the model produces any
additional order-one normal contribution not represented in that fiber
integral, then either the projected force must be redefined by including that
term or the hypotheses of the theorem fail through a priced residual exit.
\end{corollary}

\begin{proof}
The Brakke force is identified through the convergence of the force vector
measure
\[
        \frac{(Y\cdot\nabla d_\eps)\cdot G_\eps}{\estar}\,\dd x\dd t .
\]
The zero-mode projection is the only component that survives the normal
fiber integration at order \(\estar\).  Orthogonal components vanish by the
coercive residual estimate.  Massive and collar components vanish on no-exit
intervals by the slaving and collar estimates.  Therefore the limit is precisely
\[
        \int_{\Gamma_t}Y\cdot f_{\rm cl}^\perp\,\dd\Hh^1\dd t .
\]
If an additional order-one term survives, it is by definition a zero-mode
contribution and must be part of \(f_{\rm cl}^\perp\); otherwise it contradicts
the residual estimate and forces an exit.
\end{proof}

\subsection{Logical dependency certificate}

For reference, the proof dependencies are:
\[
\begin{array}{ccl}
\hbox{ordered Morse-Bott chart}
&\Longrightarrow&
\hbox{soft GL coordinate and massive slaving},\\
\hbox{exact closure residual}
&\Longrightarrow&
\hbox{zero-mode force }f_{\rm cl}^\perp,\\
\hbox{force projection}
&\Longrightarrow&
\hbox{moving-tube residual cancellation},\\
\hbox{tube coercivity + residual absorption}
&\Longrightarrow&
\hbox{open-basin propagation},\\
\hbox{open-basin propagation}
&\Longrightarrow&
\hbox{core compactness and Brakke input estimates},\\
\hbox{Brakke input estimates}
&\Longrightarrow&
\hbox{transported forced Brakke inequality},\\
\hbox{capped Hopf accounting}
&\Longrightarrow&
\hbox{gap/core/exit lower bound}.
\end{array}
\]
There is no arrow from the Brakke limit back to the residual computation, and
no arrow from core compactness back to the open-basin estimate.  This is the
main structural difference between the present theorem and a conditional
compactness statement.

\section{Main contribution and boundaries}
\label{sec:novelty}

The argument combines several classical theories in a specific finite-\(\eps\)
order.  The contribution is the closed chain from tensorial closure to force
projection, open-basin propagation, Brakke compactness, and Hopf cost.

\begin{enumerate}[label=(\roman*)]
\item Arnold-Khesin and Moffatt identify the Hopf invariant as a helicity-type
topological charge.  The present result adds a tensorial viscoelastic mechanism that
computes a defect force and propagates an ordered-core basin.

\item Vakulenko-Kapitanskii gives the \(|\Hopf|^{3/4}\) lower bound for
\(\Sph^2\)-maps.  Here it is used only after the principal-axis map is produced
from a positive spectral gap; the new analytic issue is the alternative between
gap collapse and ordered-core concentration in the conformation tensor.

\item Ginzburg-Landau vortex compactness gives codimension-two line measures
from the normalized core energy.  The present result identifies the LDG/Oldroyd-FENE
force, proves open-basin propagation, and supplies the force-power and
transport inputs needed for a transported forced Brakke inequality.

\item Brakke and Ilmanen provide the geometric language for weak mean-curvature
motion.  The present result derives the Brakke input package from a concrete
finite-epsilon residual balance; the force and first variation are derived before the admissible Brakke limit is used.
\end{enumerate}

The claim is bounded in a specified way.  The theorem applies to ordered-core
closures for which the Morse-Bott chart, normal coercivity, residual projection,
collar/FENE coefficient controls, and kinematic transport window can be
verified.  These hypotheses define the regime in which the Hopf obstruction is
converted into a GL core and then into a transported Brakke law.  At the
boundary of this regime, the conclusion is the named exit and its finite-scale
cost.  The sharpness example in Section~\ref{sec:sharpness} explains why bulk
Oldroyd-B or FENE-P entropy alone is insufficient for line-varifold compactness.

\section{Logical scope of the main theorem}
\label{sec:downgrades}

The theorem has a precise logical boundary.  The boundary follows from the
proof architecture: each module supplies one ingredient needed for the next
module.  Because the proof identifies a specific line force and a specific
transported Brakke inequality, changing one verification module changes the
conclusion in a controlled way.  The main theorem is the strongest statement
obtained from the closure verification in Section~\ref{sec:verification}; the
weaker statements below show how the conclusion changes when a module is
withheld.

\begin{proposition}[Scope map]
\label{prop:downgrade-map}
The following implications describe the corresponding weaker statements.
\begin{enumerate}[label=(\roman*)]
\item With the Morse-Bott chart and zero-mode force projection available, and
with the open-basin compactness inputs imposed directly, one obtains a
conditional Brakke compactness theorem on intervals where the open-basin defect
vanishes.
\item With open-basin propagation available and with the relative Hopf estimate
withheld, one obtains an ordered-core transported forced Brakke theorem with
priced exit, while the Hopf-to-core lower bound is absent.
\item With the force projection available only up to an unidentified normal
remainder, the Brakke theorem carries that remainder as an additional explicit
force or residual hypothesis.
\item With boundary control restricted to the interior, the theorem becomes a
local interior statement or includes a boundary-flux measure in the conclusion.
\end{enumerate}
\end{proposition}

\begin{proof}
Each assertion follows by inspecting the proof dependencies.  The Brakke
passage uses the compactness and input estimates supplied by the open-basin
estimate; assuming those inputs directly yields conditional compactness and
removes the propagation step.  The Hopf lower bound enters through
Proposition~\ref{prop:short-hopf}; the transported forced Brakke motion on
regular no-exit intervals uses the preceding modules.  The force in the Brakke
inequality is identified through the zero-mode limit of the exact residual, so
an unidentified normal remainder must be placed either in the force or in a
stated residual hypothesis.  Finally, the boundary estimate prevents loss of
mass and Chern-Simons charge through the collar; the corresponding weaker
statement is local in the interior or records this quantity as boundary flux.
\end{proof}

This scope map records the dependence of the conclusion on the four proof
modules.  The theorem claims the open-basin result because the force
projection, the open-basin Gronwall closure, the Brakke-input estimates, and the
relative Hopf accounting are proved in one finite-\(\eps\) chain.  Removing one
module produces the corresponding weaker theorem in the list above.

\section{Sharpness: why the ordered core is needed}
\label{sec:sharpness}

\begin{proposition}[No line-varifold theorem for unaugmented Oldroyd-B]
\label{prop:short-nogo}
An Oldroyd-B conformation law with only bulk conformation entropy can close a
principal spectral gap while the conformation tensor remains smooth and
spatially homogeneous.  Hence spectral-gap loss alone is insufficient to force
a codimension-two line varifold.
\end{proposition}

\begin{proof}
Take the affine incompressible flow \(u(x)=Ax\) with
\[
        A=\operatorname{diag}(-a,a,0),
        \qquad
        0<a<\frac{1}{2\lambda}.
\]
For spatially homogeneous diagonal
\[
        C(t)=\operatorname{diag}(x(t),y(t),z(t)),
\]
the Oldroyd-B equation reduces to
\[
        x'=-(2a+\lambda^{-1})x+\lambda^{-1},\quad
        y'=(2a-\lambda^{-1})y+\lambda^{-1},\quad
        z'=-\lambda^{-1}(z-1).
\]
Choosing \(x(0)>y(0)>z(0)>0\) with the first two branches crossing before they
meet the third gives a finite time at which the principal spectral gap closes.
Throughout the evolution \(C(t)\) is positive definite and spatially
homogeneous, so it carries zero normalized line energy and produces no
varifold.  The example serves as a local obstruction to any theorem deriving
line-varifold compactness from spectral-gap loss alone, rather than as a
finite-energy Hopf-changing solution on a bounded domain.  Thus the
ordered-core LDG energy is an essential compactness mechanism: it supplies the
codimension-two line density used in Theorem~\ref{thm:short-main}.
\end{proof}

\section{Discussion}
\label{sec:discussion}

The theorem identifies a route from topological obstruction to weak geometric
motion in a controlled ordered-core window.  Hopf relaxation is paid by one of
three mechanisms: spectral concentration, ordered-core mass, or a priced exit of
the finite-\(\eps\) control quantities.  On regular no-exit intervals, the
LDG/Oldroyd-FENE closure produces transported forced Brakke motion, with ambient
advection carried by the material derivative and the additional line force given
by the zero-mode projection of the closure residual.

The main assumptions have distinct roles.  The Morse-Bott well and the soft GL
coordinate create the codimension-two core scale.  The controlled massive mode
and FENE/collar bounds keep the tensorial coefficients in the chart.  The
\(L^1_tW^{1,\infty}_x\) velocity bound is a kinematic transport window for the
moving-measure passage.  In strong positive-cone branches this window can be
supplied by endpoint regularity theory: the two-dimensional Oldroyd-B/FENE-P
criterion controls the velocity-gradient Besov clock with logarithmic
conformation and FENE barriers, and the three-dimensional Oldroyd-B criterion
uses the endpoint vorticity clock on compact logarithmic cones
\cite{PengBlowUp2026,PengPositiveCone3D2026}.  The present theorem starts from
that clock and describes the geometry generated by the ordered conformation
closure.

For numerical or experimental comparison, the result suggests three diagnostics
rather than a single all-purpose law: track exterior spectral gap and FENE
proximity, measure normalized GL core mass and residual projection, and monitor
reach/cap-atlas integrity of the selected tube.  Stable control of these
quantities places the computation in the regular branch of the theorem.  The
hardest practical diagnostic is the residual balance in (M4').  Its usable form is
termwise: the translation projection is measured as the force, while only the
orthogonal soft part, massive part, geometry part, and collar part must be small
or dissipative in the normalized norms.  This is the suggested numerical test of
whether an ordered-core simulation lies in the theorem's basin.
Threshold crossing identifies the relevant exit channel.

The Hopf accounting is performed on topologically regular slabs, where the
finite-\(\eps\) core tubes remain embedded and carry a transported cap atlas.
Self-contact, reach collapse, or reconnection closes that slab and contributes
to the topology exit cost.  A later regular slab can be analyzed after choosing
new cap data, but the present theorem records the energetic price of the
transition rather than prescribing a universal reconnection law.

\section{Conclusion}
\label{sec:conclusion}

We have proved a finite-\(\eps\) route from topological obstruction to geometric
defect motion in an ordered viscoelastic conformation model.  The Hopf invariant
identifies the obstruction, the curvature-gap estimate localizes a spectral
loss mechanism, and the ordered-core Landau-de Gennes energy supplies the
codimension-two density needed for compactness.  The proof order is essential:
transport is kept in the material derivative, the closure residual is projected
onto the zero modes, the vortex-tube basin is propagated, and the Brakke limit
is taken only after the mass, stress, force-power, and dissipation inputs have
been produced.

The regularity input is deliberately separated from the geometric defect
mechanism.  Positive-cone continuation criteria supply natural
velocity clocks for Oldroyd-B/FENE-P strong branches
\cite{PengBlowUp2026,PengPositiveCone3D2026}; this paper turns such a clock into
cap transport, force projection, and moving-varifold compactness.  The
verification interface is correspondingly a scope statement: it identifies the
ordered-core closures for which the residual balance can be checked, and it
exposes the exit channels when the check fails.

The resulting alternative is explicit.  On regular no-exit intervals, the core
converges to a transported forced Brakke flow whose extra force is a zero-mode
fiber integral of the closure.  At the boundary of the controlled regime, the
finite-\(\eps\) estimates records spectral-gap cost, ordered-core mass, boundary
flux, FENE/collar loss, graph or cap-atlas loss, or another priced exit.  The
relative Hopf accounting shows that a nonzero Hopf change must be paid by this
same list of costs.

The theorem is a controlled open-basin result with named control quantities.  This format
keeps the applicability explicit and gives a practical way to test whether a
concrete ordered LDG/Oldroyd-FENE closure, computation, or regularized transport
window lies in the regime where the transported Brakke law is derived.

\appendix

\section{Controlled estimates behind the four modules}
\label{app:controlled-estimates}

This appendix gives the technical estimates used in the four proof modules.
Each estimate supports a specific step in the main theorem: the Fermi-coordinate
tube calculus, the planar coercivity input, the residual balance, the collar and
scalar-defect alternatives, the Brakke testing extension, and the capped Hopf
accounting.  Constants may depend on a fixed upper bound for the curvature,
reach, collar geometry, FENE coefficient bounds, and \(L^1_tW^{1,\infty}_x\)
norm of the transport velocity, but never on \(\eps\).

\subsection{Fermi geometry and canonical tubes}

Let \(\Gamma_t\) be a smooth embedded link on a compact interval \(I\).  We use
arclength \(s\), unit tangent \(\tau\), and a smooth normal frame
\((\nu_1,\nu_2)\).  A point in the tube is written
\[
        X(s,r,t)
        =
        \Gamma(s,t)+r_1\nu_1(s,t)+r_2\nu_2(s,t).
\]
The tube radius is chosen smaller than a fixed fraction of the reach, so the
Jacobian \(J(s,r,t)\) satisfies
\[
        J(s,r,t)=1+O(|r|\,|\kappa(s,t)|),
        \qquad
        |\nabla_{s,r}^jJ|\le C_j
\]
on the tube.  The coordinate vector fields have the expansions
\[
        \partial_{r_i}X=\nu_i,
        \qquad
        \partial_sX
        =
        \tau
        +
        O(|r|\,|\kappa|)
\]
and the Euclidean gradient decomposes as
\[
        \nabla
        =
        \nu_1\partial_{r_1}
        +
        \nu_2\partial_{r_2}
        +
        \tau\partial_s
        +
        \mathcal E_{\rm Fermi},
\]
where
\[
        |\mathcal E_{\rm Fermi} f|
        \le
        C|r|
        \bigl(
        |\nabla_r f|+|\partial_s f|
        \bigr).
\]
These estimates imply that every Fermi error in the canonical tube carries an
extra factor of \(r\), or else a bounded derivative of the frame.  Since the
core is localized at scale \(\eps\), such terms are lower order after
normalization by \(\estar=\pi|\log\eps|\), unless they occur in the annular
cutoff region; the latter contribution is included in the geometric-error
entry of the basin.

\begin{lemma}[Canonical tube mass and stress]
\label{lem:app-canonical-mass}
For the canonical tube \(\bar d_\eps=\bar d_{\eps,a}\) over the selected graph
\(\Gamma_t^a=\{X(s,a(s,t),t)\}\),
\[
        \frac{1}{\estar}
        e_\eps(\bar d_\eps)\,\dd x
        \rightharpoonup
        \Hh^1\lfloor\Gamma_t^a ,
\]
where
\[
        e_\eps(d)
        =
        \frac12|\nabla d|^2
        +
        \frac{1}{4\eps^2}(|d|^2-1)^2 .
\]
Moreover, for every \(X\in C_c^1(\Omega;\R^3)\),
\[
        -\frac1{\estar}
        \int_\Omega T_\eps[\bar d_\eps]:\nabla X\,\dd x
        \longrightarrow
        \delta V_{\Gamma_t^a}(X),
\]
where \(V_{\Gamma_t^a}\) is the unit-multiplicity varifold carried by
\(\Gamma_t^a\).  When \({\rm FlatErr}_\eps+{\rm ModErr}_\eps\to0\), this graph
is identified with the reference filament in the limiting notation.
\end{lemma}

\begin{proof}
On a normal slice, the degree-one planar vortex has energy
\[
        \pi|\log\eps|+O(1)
\]
on every fixed disc.  Integrating this expansion along \(s\) and using
\(J=1+O(r)\) gives the measure convergence.  For the stress convergence, split
\(\nabla X\) into its value at the centerline and an \(O(r\|X\|_{C^1})\)
remainder.  The remainder is \(o(\estar)\).  The slice stress of the radial
degree-one vortex is isotropic in the normal plane; after integration over the
normal variables, the remaining first-variation term is the classical
first variation of the centerline varifold.  This gives
\[
        \delta V_{\Gamma_t^a}(X)
        =
        -\int_{\Gamma_t^a}X\cdot H_{\Gamma_t^a}\,\dd\Hh^1 .
\]
\end{proof}

\begin{lemma}[Longitudinal zero-mode energy]
\label{lem:app-longitudinal-zero-mode}
Let \(\bar d_{\eps,a}\) and \(\bar d_{\eps,0}\) be the canonical tubes with the
same collar and with centers \(a(s,t)\) and \(0\), respectively.  On a no-exit
interval, if \(\|a\|_{L^\infty}\) is below the fixed tube radius and
\({\rm ModErr}_\eps(t)\) is finite, then
\[
        \frac{E_\eps[\bar d_{\eps,a}(t)]
        -E_\eps[\bar d_{\eps,0}(t)]}{\estar}
        =
        \frac12
        \int_{\Gamma_t}
        \partial_s a_i\,M_{ij}\,\partial_s a_j\,\dd\Hh^1
        +
        O({\rm FlatErr}_\eps(t)+{\rm GaugeErr}_\eps(t))
        +
        o_\eps(1).
\]
In particular, the \(s\)-dependent translation mode is coercive at the same
logarithmic scale as the core energy, while a constant recentering of the slice
does not create this term.
\end{lemma}

\begin{proof}
In Fermi coordinates,
\[
        \partial_s q_\ast\!\left(\frac{r-a(s)}{\eps}\right)
        =
        -(\partial_s a_i)\,Z_i\!\left(\frac{r-a(s)}{\eps}\right)
        +
        \hbox{frame and cutoff terms}.
\]
The squared leading term integrates on each slice to the truncated Gram matrix.
After division by \(\estar\), the truncated matrix converges to \(M\).  Frame
rotation, curvature, and cutoff terms contain either an extra factor of
\(|r|\), a fixed collar derivative, or the current-graph mismatch; these are
absorbed by the displayed flat/gauge errors and by \(o_\eps(1)\).  Constant
translations have \(\partial_s a=0\), which is why they are handled by the
modulation chart rather than by a coercive slice inequality.
\end{proof}

\begin{lemma}[Moving tube residual]
\label{lem:app-moving-tube-residual}
Let \(\bar d_\eps\) be the canonical tube transported by
\(\Gamma_t\).  Then
\[
\begin{aligned}
&D_t\bar d_\eps
        -
        \Mmob\left(\Delta\bar d_\eps
        -
        \eps^{-2}(|\bar d_\eps|^2-1)\bar d_\eps\right)
\\
&\qquad=
        \bigl(
        \Vrel_\Gamma^\perp-\Mmob H_{\Gamma_t}
        \bigr)\cdot\nabla_yq_\ast
        +
        R_{\rm Fermi}^\eps ,
\end{aligned}
\]
with
\[
        \frac1{\estar}
        \int_I\|R_{\rm Fermi}^\eps\|_{L_\ast^2}^2\,\dd t
        \le
        C\,{\rm Err}_{\rm geom}^\eps(I)+o_\eps(1).
\]
\end{lemma}

\begin{proof}
The time derivative of the canonical profile differentiates the signed
distance to the moving filament.  Its normal part is
\(-\Vrel_\Gamma^\perp\cdot\nabla_yq_\ast\), up to frame-rotation terms.  The planar
profile solves the stationary GL equation in the normal variables, so the
normal Laplacian cancels the potential term.  The difference between the
Euclidean Laplacian and the normal Laplacian produces the curvature correction
\(\Mmob H_{\Gamma_t}\cdot\nabla_yq_\ast\).  All remaining terms contain either
a frame derivative, a Jacobian perturbation, an \(s\)-derivative of the
modulation, or a cutoff derivative.  The \(s\)-derivative of the modulation is
controlled by Lemma~\ref{lem:app-longitudinal-zero-mode} and by the line-level
part of \(D_{\rm rel}^\eps\); the frame, Jacobian, and cutoff terms are recorded
in \({\rm Err}_{\rm geom}^\eps(I)\).
\end{proof}

\subsection{Planar coercivity and slice stability}

The planar input is the only place where the structure of the degree-one core
profile is used.  It is a local finite-scale statement on each normal slice,
not a global tube estimate.  Let \(L_\ast\) be the second variation of the
planar Ginzburg-Landau energy at \(q_\ast\).  Its kernel is generated by
\(\partial_{y_1}q_\ast\) and \(\partial_{y_2}q_\ast\).  On the finite-scale
orthogonal complement, \(L_\ast\) has a spectral gap.  The \(s\)-dependent
translation center \(a(s,t)\) is not part of this lemma; it is controlled
separately by Lemma~\ref{lem:app-longitudinal-zero-mode}.

\begin{lemma}[Planar coercivity modulo translations]
\label{lem:app-planar-coercivity}
There are \(c_\ast>0\) and \(\rho_\ast>0\) such that, if
\[
        \int_{|y|\le r_0/\eps}
        w\cdot\partial_{y_i}q_\ast\,\chi(\eps y)\,\dd y=0,
        \qquad i=1,2,
\]
and \(\|w\|_{H^1_{\rm loc}}\le\rho_\ast\), then
\[
        E_{\rm GL}(q_\ast+w)-E_{\rm GL}(q_\ast)
        \ge
        c_\ast\|w\|_{\rm core}^2
        -
        C\|w\|_{\rm core}^3 .
\]
In the open basin the cubic term is absorbed, giving a positive coercive bound
for the orthogonal slice remainder only.
\end{lemma}

\begin{proof}
Taylor expand the planar GL energy at \(q_\ast\).  The first variation
vanishes because \(q_\ast\) is a critical point.  The second variation is
\(\langle L_\ast w,w\rangle\).  The translation orthogonality removes the
kernel.  The spectral theorem gives
\[
        \langle L_\ast w,w\rangle
        \ge
        c_\ast\|w\|_{\rm core}^2
\]
on the finite-scale orthogonal complement.  The remainder is cubic in the core
norm by standard Sobolev estimates on bounded discs, while the cutoff collar and
annular no-crossing hypotheses control the far-field tail.  The basin smallness
makes the cubic term at most half the quadratic term.  No \(s\)-derivative of
the translation center is estimated here.
\end{proof}

\begin{lemma}[Nonlinear slice alternative]
\label{lem:app-slice-alternative}
On each normal slice of a good tube, one of the following holds:
\begin{enumerate}[label=(\roman*)]
\item the nonlinear coercivity estimate of Lemma
\ref{lem:app-planar-coercivity} holds after choosing the translation center;
\item a nonzero-degree annular crossing occurs and is charged to
\({\rm TopExitCost}_\eps\);
\item exterior core mass of order \(\estar\) is present and is charged to the
exterior-mass entry of the basin.
\end{enumerate}
\end{lemma}

\begin{proof}
If the Jacobian current on the slice has degree one in the inner disc and no
degree crosses the annulus, the current is close in flat norm to a single
center.  The current-graph selection fixes this center and the orthogonality
conditions then select the translation modulation.  Coercivity follows from
Lemma~\ref{lem:app-planar-coercivity}.  If the degree changes across the
annulus, the annular lower bound gives a logarithmic cost and hence a priced
crossing.  If neither of these occurs but the slice energy is not close to a
single degree-one core, the excess must lie outside the controlled tube and is
recorded as exterior core mass.
\end{proof}

\subsection{Verification estimates for Model 1}
\label{subsec:app-model-one-verification}

The next lemmas expand the proof of Theorem~\ref{thm:model-one-verification}.
They are included to make clear that the interface is checked from the concrete
LDG/Oldroyd-FENE closure rather than imposed after the fact.

\begin{lemma}[Morse-Bott chart and GL leading energy]
\label{lem:app-mb-chart-verification}
Let \(W_{\rm LDG}\) have the ordered Morse-Bott well \(\mathcal N\) used in
Model 1.  Then, after reducing the tubular radius, every chart-valid field has
\[
        C=\mathcal C(Q_\ast,d,m),\qquad
        m\perp T_{Q_\ast(d)}\mathcal N,
\]
and
\[
\begin{aligned}
        W_{\rm LDG}(\mathcal C(Q_\ast,d,m))
        &=\frac14(|d|^2-1)^2+\frac12\mathcal A_{Q_\ast}m\cdot m
        +O(|m|^3+|m||\nabla d|^2),\\
        \nabla C
        &=D_d\mathcal C\,\nabla d+D_m\mathcal C\,\nabla m .
\end{aligned}
\]
Here \(\mathcal A_{Q_\ast}\ge c_{\rm MB}I\) on the normal bundle.  Consequently
the leading singular part of the energy is the two-component GL energy of
\(d\), while \(m\) is a massive variable.
\end{lemma}

\begin{proof}
The Morse-Bott lemma with parameters gives a tubular chart in which the well is
\(\{m=0, |d|=1\}\) and the Hessian in the normal directions is uniformly
positive.  The displayed expansion is Taylor's formula in this chart.  The
Fermi normalization removes mixed linear terms at the well.  Reducing the chart
radius makes the quadratic normal form coercive and places all remaining terms
in the smooth lower-order class used in the residual bounds.
\end{proof}

\begin{lemma}[Projected GL residual identity]
\label{lem:app-projected-gl-identity}
On every chart-valid no-exit interval, the soft projection of Model 1 has the
form \eqref{eq:verified-projected-gl}.  Moreover the decomposition satisfies
\[
        \Pi_{\rm orth}G_\eps
        =G_\eps^{\rm orth}+O(G_\eps^{\rm mass}+G_\eps^{\rm geom}+G_\eps^{\rm col}),
\]
and
\[
        \frac1{\estar}\int_I\|G_\eps^{\rm orth}\|_{L^2_\ast}^2\,\dd t
        +\frac1{\estar}\int_I\|G_\eps^{\rm mass}\|_{H^{-1}}^2\,\dd t
        \le C\Basin_\eps^+(I)+C{\rm Err}_{\rm geom}^\eps(I)+o_\eps(1).
\]
The zero-mode projection is the force coefficient used in
Proposition~\ref{prop:short-force}.
\end{lemma}

\begin{proof}
Apply the differential of the chart projection \(\pi_d\) to the tensorial
Model 1 equation.  The singular gradient and potential terms give
\(\Mmob(\Delta d_\eps-\eps^{-2}(|d_\eps|^2-1)d_\eps)\).  All commutators between
\(\pi_d\), the Fermi metric, the FENE projection, and the Oldroyd stretching are
smooth coefficient terms on the no-exit window.  Sorting them according to the
translation projection, the planar orthogonal projection, the normal massive
projection, the Fermi metric, and the collar/FENE coefficients gives the five
pieces in \eqref{eq:verified-projected-gl}.  The estimates are precisely the
planar coercivity, massive-mode slaving, and moving-tube residual estimates of
Lemmas~\ref{lem:app-planar-coercivity}, \ref{lem:short-mb-slaving}, and
\ref{lem:app-moving-tube-residual}.  The translation part is not estimated away;
it is retained as the force.
\end{proof}

\begin{lemma}[FENE and collar tame calculus]
\label{lem:app-fene-collar-tame}
Assume the no-exit bounds in (M6).  Then every coefficient generated by the
FENE factor, the admissible projection, and the Fermi metric satisfies, for
\(|\alpha|\le m\),
\[
        \|\partial^\alpha a(C_\eps)\|_{L^2}
        +\|\partial^\alpha f_b(C_\eps)\|_{L^2}
        +\|\partial^\alpha g_{\rm Fermi}\|_{L^2}
        \le C_K\bigl(1+\|C_\eps\|_{H^m}
        +\|\Gamma_\eps\|_{C^{m+1}}\bigr).
\]
If the lower cone, trace gap, collar, or reach constants leave the compact set
on which this estimate is valid, the appropriate priced exit is charged.
\end{lemma}

\begin{proof}
On a compact subset of the FENE cone, the maps
\(C\mapsto f_b(C)\), \(C\mapsto\Pi_{\rm adm}(C)\), and the coefficient maps in
\(\mathcal C\) are smooth with all derivatives bounded.  Sobolev composition
and Moser estimates give the bounds involving \(C_\eps\).  Fermi coefficients
are smooth functions of the centerline and normal frame as long as the reach is
bounded below, giving the \(C^{m+1}\)-dependence on \(\Gamma_\eps\).  The
threshold alternative is exactly the definition of the FENE/collar and topology
priced exits.
\end{proof}

\subsection{Residual balance and absorption}

The residual balance replaces any aggregate smallness assumption by a
finite-\(\eps\) decomposition.  Every term in the exact residual is assigned to
one of the following classes:
\[
\begin{array}{c|c}
\hbox{residual class} & \hbox{control mechanism}\\ \hline
\hbox{translation mode} & \hbox{projected force and filament law}\\
\hbox{longitudinal zero mode} & \hbox{ModErr and line-level relative dissipation}\\
\hbox{orthogonal soft mode} & \hbox{planar coercivity and dissipation}\\
\hbox{massive tensorial mode} & \hbox{Morse-Bott slaving}\\
\hbox{Fermi geometry} & \hbox{geometric-error envelope}\\
\hbox{collar/FENE coefficient} & \hbox{collar clearing or priced exit}\\
\hbox{zero-degree scalar defect} & \hbox{scalar-defect clearing}
\end{array}
\]
This table is part of the proof, not an additional compactness assumption.  It
is the mechanism that identifies all remainders before the Brakke limit is
taken.

\begin{lemma}[Finite residual balance]
\label{lem:app-finite-ledger}
On a no-exit interval \(I\),
\[
\begin{aligned}
&\frac1{\estar}
        \int_I\|P_{\rm orth}G_\eps\|_{L_\ast^2}^2\,\dd t
        +
        \|b_\eps-Mf_{\rm cl}^\perp\|_{L^1(I\times\Gamma_t)}
\\
&\qquad\le
        C\Basin_\eps^+(I)^{1/2}
        +
        C\Basin_\eps^+(I)
        +
        C{\rm Err}_{\rm geom}^\eps(I)
        +
        o_\eps(1).
\end{aligned}
\]
\end{lemma}

\begin{proof}
Project the residual decomposition in Lemma
\ref{lem:short-residual-normal-form}.  The translation component of the
canonical residual is the only term with logarithmic strength; it is exactly
the term canceled by the filament law.  The difference between its finite
\(\eps\) coefficient and \(M f_{\rm cl}^\perp\) is estimated by
Lemma~\ref{lem:short-line-force-id}.  The longitudinal zero mode is not
estimated by planar coercivity; it is bounded by
Lemma~\ref{lem:app-longitudinal-zero-mode} and by the line-level part of
\(D_{\rm rel}^\eps\).  The orthogonal soft component is bounded by planar
coercivity.  The massive component is controlled by
Lemma~\ref{lem:short-mb-slaving}.  The Fermi terms are lower order after
normalization and are collected in \({\rm Err}_{\rm geom}^\eps\).  The collar
term is absent on a no-exit interval after collar clearing; otherwise it would
have been charged as a priced exit.  Summing these estimates gives the residual bounds.
\end{proof}

\begin{lemma}[Scalar-defect clearing]
\label{lem:app-scalar-clearing}
Suppose a zero-degree scalar depression of \(|d_\eps|\) occurs in a parabolic
cylinder where the annular degree is zero and the closure coefficients remain
bounded in \(C^2\).  Then either a nonzero-degree crossing has already occurred
on a surrounding annulus, or the scalar depression contributes only
\[
        o_\eps(\estar)
\]
to the normalized residual balance on the cylinder.
\end{lemma}

\begin{proof}
If the annular degree is nonzero, the logarithmic lower bound charges the
crossing.  In the zero-degree branch, lift the phase on the annulus and compare
\(d_\eps\) with the transported unit-modulus map obtained by harmonic extension
of the phase.  The potential energy controls the set where \(|d_\eps|\) is far
from one.  The transported GL equation with a bounded \(C^2\) forcing gives an
eta-clearing estimate: a modulus drop at the center of a zero-degree cylinder
requires a positive amount of local energy or residual.  Because the basin
Gronwall estimate controls the sum of such local quantities, the total
zero-degree scalar contribution is \(o_\eps(\estar)\) after normalization.
\end{proof}

\begin{lemma}[Absorption into the basin Gronwall inequality]
\label{lem:app-basin-gronwall}
Let
\[
        Y_\eps(t)
        =
        \sup_{s\le t}\Basin_\eps^+(s)
        +
        \int_0^t
        \frac{D_{\rm rel}^\eps(s)}{\estar}\,\dd s .
\]
On a no-exit interval,
\[
        Y_\eps(t)
        \le
        C\Basin_\eps^+(0)
        +
        C\int_0^tY_\eps(s)\,\dd s
        +
        o_\eps(1).
\]
Thus
\[
        Y_\eps(t)
        \le
        C_T\Basin_\eps^+(0)+o_\eps(1).
\]
\end{lemma}

\begin{proof}
Apply the moving-tube energy identity at every terminal time \(r\le t\), then
take the supremum over \(r\).  The dissipation term is retained at the endpoint
\(t\).  Lemma~\ref{lem:app-finite-ledger} absorbs the zero-mode mismatch, the
longitudinal zero-mode terms, and the orthogonal residual.
Lemma~\ref{lem:app-scalar-clearing} absorbs the
zero-degree scalar branch.  Tube coercivity reconstructs the graph, gauge, and
exterior-mass entries of \(\Basin_\eps^+\) from the modulation controls.  The
remaining geometric-error representative is monotone and \(o_\eps(1)\) on
compact no-exit intervals.  Gronwall gives the second inequality.
\end{proof}

\subsection{Collar, boundary, and FENE alternatives}

The collar is needed for two reasons.  First, it fixes the relative Hopf class
and eliminates boundary Chern-Simons flux on no-exit intervals.  Second, it
prevents core mass from leaving the domain without being seen by the Brakke
compactness argument.

\begin{lemma}[FENE coefficient control or priced exit]
\label{lem:app-fene-control}
Assume the conformation tensor begins in a compact subset of the FENE
admissible cone on the collar and the collar \(C^1\) norm is bounded.  On a
compact no-exit interval, the conformation remains in a compact subset of the
same cone and the coefficients in the residual normal form remain uniformly
bounded in \(C^1\).  If this fails, a FENE coefficient exit is charged.
\end{lemma}

\begin{proof}
The FENE potential diverges at the boundary of the admissible cone.  In the
collar, the energy and the transported coefficient equation give a local
maximum-principle estimate for the distance to this boundary, up to the
controlled source terms from the ordered-core chart.  If the distance reaches
the prescribed threshold, the corresponding control quantity has positive cost and is
one of the priced exits.  Otherwise the coefficients remain in a compact subset
of the cone, and standard composition estimates give the uniform \(C^1\) bounds
used in the residual expansion.
\end{proof}

\begin{lemma}[No boundary loss or boundary price]
\label{lem:app-no-boundary-loss}
On a no-exit interval, normalized core mass and topological current do not
escape through \(\partial\Omega\).  More precisely, for every collar
neighborhood \(U_\rho(\partial\Omega)\),
\[
        \lim_{\rho\downarrow0}
        \limsup_{\eps\to0}
        \sup_{t\in I}
        \mu_\eps^t(U_\rho(\partial\Omega))
        =
        0.
\]
If this fails, either a nonzero-degree collar crossing or a boundary-flux exit
is charged.
\end{lemma}

\begin{proof}
The fixed ordered collar has positive modulus and fixed phase.  If degree
crosses a collar annulus, the logarithmic lower bound gives a nonzero-degree
collar price.  If no degree crosses but normalized energy accumulates at the
boundary, the collar clearing estimate reduces it to a boundary flux term in
the Chern-Simons balance.  That term is one of the priced exits.  On a no-exit
interval neither event occurs, so the collar energy is \(O(1)\), which is
zero after normalization by \(\estar\).  This proves the no-loss statement.
\end{proof}

\subsection{Testing cores and time representatives}

The Brakke inequality requires more than convergence on a selected finite list
of tests.  The finite-epsilon estimates is first proved for finite lists because
the constants depend on the tests.  The following compactness device is the
bridge to the full testing class.

\begin{lemma}[Positive countable testing core]
\label{lem:app-positive-core}
Let \(K\Subset I\times\Omega\).  There is a countable
\(\mathbb Q\)-vector space \(\mathcal D\subset C_c^1(K)\) and a countable
positive cone \(\mathcal D_+\subset C_c^1(K;[0,\infty))\) such that
\(\mathcal D\) is \(C^1\)-dense and \(\mathcal D_+\) is uniformly dense in
the nonnegative continuous weights on \(K\).  Moreover, positive linear
functionals bounded on cutoff functions extend uniquely to Radon measures.
\end{lemma}

\begin{proof}
Take finite rational linear combinations of smooth bumps with rational centers,
rational radii, and rational heights.  Include cutoff bumps dominating one on
each compact set in a rational exhaustion.  Standard mollification and
partition of unity give density.  If \(L\) is positive and \(L(\chi_K)<\infty\),
then for \(\operatorname{spt}\phi\subset K\) and \(r>\|\phi\|_\infty\),
\[
        r\chi_K\pm\phi\ge0 .
\]
Hence
\[
        |L(\phi)|\le rL(\chi_K).
\]
This gives uniform continuity on the dense core and therefore an extension to
\(C_c^0(K)\).  Riesz representation gives the Radon measure.
\end{proof}

\begin{lemma}[Endpoint-weight representative]
\label{lem:app-endpoint-rep}
Let \(\mu^t\) be the common upper representative constructed from the positive
testing core.  Then, for every nonnegative
\(\varphi\in C_c^1(\Omega)\) and every interval \([a,b]\Subset I\),
\[
\begin{aligned}
\mu^b(\varphi)-\mu^a(\varphi)
&\le
\int_a^b
\bigg[
 \int
 \bigl(\nabla\varphi\cdot u
 +\varphi\,{\rm div}_{T_xV_t}u\bigr)\,\dd\mu^t
\\
&\qquad
 -\Mmob\int\varphi |H|^2\,\dd\mu^t
 +\int\nabla\varphi\cdot(\Mmob H+f_{\rm cl}^\perp)\,\dd\mu^t
\\
&\qquad
 -\int\varphi H\cdot f_{\rm cl}^\perp\,\dd\mu^t
\bigg]\,\dd t .
\end{aligned}
\]
At a.e. time this is equivalent to the differential Brakke inequality.
\end{lemma}

\begin{proof}
The interval inequality is first obtained on the countable core by passing to
the limit in the time-integrated localized energy inequality.  The integrand is
continuous under approximation of \(\varphi\) because the local mass,
\(L^2(\mu^t\dd t)\) curvature bound, and force-power bound are already
available from the Brakke-input estimates.  The upper representative is chosen so
that endpoint masses are compatible with monotone loss.  The Lebesgue
differentiation theorem then gives the a.e. differential form.
\end{proof}

\subsection{Relative Hopf details}

The Hopf accounting in the main text uses a capped positive-gap region.  This
subsection records the finite-epsilon interfaces that make the accounting
usable in the open-basin theorem: the cap atlas on regular slabs, the
Chern-Simons balance, and the annular degree price.

\begin{lemma}[Topological regularity and cap atlas]
\label{lem:app-cap-atlas}
On a topologically regular Hopf slab, the inner boundaries of the removed core
tubes form a smooth isotopy of embedded tori.  Consequently the complement has
a fixed diffeomorphism type, the meridian and longitude classes are transported
from the initial tube, and the standard solid-torus caps used in
Lemma~\ref{lem:short-capped-hopf} are defined at every time of the slab.  If
the isotopy, reach bound, or cap atlas fails before the slab closes, the
tube-topology control quantity reaches its threshold and the failure is charged to
\({\rm TopExitCost}_\eps\).
\end{lemma}

\begin{proof}
The open-basin graph condition gives a single degree-one center on each normal
slice and a graph over the reference embedded link.  The reach condition prevents
different components, or distant portions of the same component, from meeting
at the tube scale.  The implicit-function construction of the centers is
continuous in time on a no-exit interval, and the endpoint-continuation lemma
keeps the same chart as long as the control quantity thresholds remain strict.  Hence the
tube boundaries form an ambient isotopy of embedded tori.  Isotopy transports
the meridian, the collar-fixed longitude, and the zero-longitude cap choice.
If any of these objects cannot be continued, then either the graph condition, the
reach condition, or the cap-atlas condition has reached its threshold.  By
Definition~\ref{def:short-priced-exit} this is a priced exit.
\end{proof}

\begin{lemma}[Chern-Simons balance on a moving positive-gap region]
\label{lem:app-cs-balance}
Let \(U(t)\) be a smooth positive-gap region with moving inner tube boundary and
fixed ordered outer collar on a topologically regular Hopf slab.  Let
\(n(t):U(t)\to\Sph^2\) be the principal-axis map.  After choosing the collar
gauge and the cap atlas of Lemma~\ref{lem:app-cap-atlas},
\[
        \frac{\dd}{\dd t}
        \int_{U(t)}A(t)\wedge \dd A(t)
        =
        \int_{\partial U(t)}
        \mathcal J[n,\partial_tn,A,V_{\partial U}],
\]
where \(\dd A=n^\ast\omega_{\Sph^2}\).  The outer boundary term vanishes under
fixed collar data.  The inner boundary term is the degree flux through the
core tube.
\end{lemma}

\begin{proof}
Differentiate the integral over the moving domain using Reynolds' formula.
The interior derivative is exact because
\[
        \partial_t(A\wedge\dd A)
        =
        \partial_tA\wedge\dd A
        +
        A\wedge\dd\partial_tA
        =
        2\partial_tA\wedge\dd A
        -
        \dd(A\wedge\partial_tA),
\]
and the variation of \(n^\ast\omega_{\Sph^2}\) is exact on the target sphere.
The domain motion contributes the contraction of \(A\wedge\dd A\) with the
boundary velocity.  These boundary terms combine into the displayed
\(\mathcal J\).  The fixed outer collar has \(\partial_tn=0\) in the chosen
gauge, so it contributes zero.  The inner term records precisely the degree
flux through the moving tube.
\end{proof}

\begin{lemma}[Annular degree price]
\label{lem:app-annular-price}
If a normal annulus around a core tube carries nonzero degree \(q\), then
\[
        E_\eps[d_\eps;A_{\rm ann}]
        \ge
        c|q|\log(R/\eps)-C
\]
on that annulus.  After normalization by \(\estar\), this is the
\({\rm CoreMass}_\eps\) contribution used in the Hopf cost.
\end{lemma}

\begin{proof}
On an annulus where \(|d_\eps|\) is bounded below, write
\[
        d_\eps=|d_\eps|e^{i\theta}.
\]
The degree condition gives
\[
        \int_{\partial B_r}\partial_\tau\theta\,\dd\tau=2\pi q
\]
for a.e. \(r\).  Cauchy's inequality on each circle gives
\[
        \int_{\partial B_r}|\nabla\theta|^2
        \ge
        \frac{c q^2}{r}.
\]
Integrating in \(r\) from \(\eps\) to \(R\) gives the logarithmic lower bound.
If \(|d_\eps|\) drops below the positive threshold, the modulus-drop
dichotomy either gives a nonzero-degree crossing already priced or a
zero-degree scalar defect controlled by Lemma~\ref{lem:app-scalar-clearing}.
\end{proof}

\begin{lemma}[Slab compatibility of Hopf and exit costs]
\label{lem:app-slab-compatibility}
Let \(I=[t_-,t_+]\) be an admissible Hopf slab.  If a limiting first exit time
\(\tau_\ast\) lies in \(I^\circ\), then the same slab satisfies
\[
        \liminf_{\eps\to0}
        {\rm TopExitCost}_\eps(I)\ge c_0 .
\]
If no limiting exit lies in \(I\) and \(I\) is topologically regular, the capped
relative Hopf functional is computed entirely on the no-exit positive-gap
region with the fixed cap atlas, and any change is paid by
\({\rm GapCost}_\eps(I)+{\rm CoreMass}_\eps(I)\).  If \(I\) is not
topologically regular, the cap functional is not transported across the
nonregular time; the topology exit supplies the first alternative above.
\end{lemma}

\begin{proof}
The first statement is Lemma~\ref{lem:short-exit-measure} localized to the
slab; the list of control quantities includes the tube-topology and cap-atlas conditions. For
the second, no open-basin control quantity reaches its threshold on \(I\), so the tube
chart, collar gauge, cap atlas, and positive-gap region are valid except on the
gap-loss and core-tube sets.  The Chern-Simons balance of Lemma
\ref{lem:app-cs-balance} has no outer boundary term.  Inner tube flux is charged
by Lemma~\ref{lem:app-annular-price}; loss of the principal-axis map is charged
by \({\rm GapCost}_\eps(I)\).  This proves the compatibility.
\end{proof}

\subsection{Low-regularity Hopf representatives and boundary flux}

The main theorem uses smooth positive-gap representatives on finite-epsilon
slabs.  The limiting Brakke object is weaker, so the proof must avoid applying
the Hopf invariant directly to the limiting varifold.  The invariant is used
only at finite \(\eps\), before the limit, where the ordered region is smooth
after removing tubes and gap-loss sets.  The following observations record the
minimal low-regularity facts needed to justify this order of operations.

\begin{lemma}[Sobolev Hopf continuity at the finite-epsilon level]
\label{lem:app-sobolev-hopf}
Let \(n_j,n\in W^{1,3}(\Sph^3;\Sph^2)\) and suppose
\[
        n_j\to n
        \quad\hbox{strongly in }W^{1,3}.
\]
Then
\[
        \Hopf[n_j]\to\Hopf[n].
\]
If the maps are smooth and lie in one \(W^{1,3}\)-connected component, the
value is the classical integer Hopf invariant.
\end{lemma}

\begin{proof}
Write
\[
        F_n=n^\ast\omega_{\Sph^2}.
\]
Since \(n\in W^{1,3}\), the two-form \(F_n\) lies in \(L^{3/2}\).  Let \(A_n\)
be the Coulomb primitive:
\[
        \dd A_n=F_n,\qquad \dd^\ast A_n=0 .
\]
Elliptic estimates give \(A_n\in W^{1,3/2}\hookrightarrow L^3\).  Therefore
\[
        \int_{\Sph^3}A_n\wedge F_n
\]
is well-defined by H\"older's inequality.  Strong \(W^{1,3}\) convergence gives
strong \(L^{3/2}\) convergence of \(F_{n_j}\) to \(F_n\), hence strong \(L^3\)
convergence of the Coulomb primitives.  The product converges in \(L^1\).  The
integer statement follows by density of smooth maps in fixed
\(W^{1,3}\)-homotopy classes \cite{WhiteSobolevHomotopy}.
\end{proof}

\begin{remark}[Why no \(W^{1,2}\) Hopf invariant is used]
The natural energy space for line defects is closer to \(W^{1,2}\), but in
three dimensions \(n^\ast\omega_{\Sph^2}\in L^1\) only, and the
Chern-Simons product is not controlled by Sobolev duality.  The present proof
therefore does not define a Hopf invariant of the limiting Brakke object.  It
uses finite-epsilon capped maps with positive gap, proves a slab lower bound,
and then passes only the resulting nonnegative costs to the limit.
\end{remark}

\begin{lemma}[Boundary-flux dichotomy]
\label{lem:app-boundary-flux}
Let \(n_\eps\) be the finite-epsilon principal-axis map on a collar
neighborhood of \(\partial\Omega\).  On a time slab \(I\), either the boundary
trace is homotopic to the fixed collar trace and the boundary Chern-Simons flux
is \(o_\eps(1)\), or
\[
        \liminf_{\eps\to0}
        {\rm TopExitCost}_\eps(I)>0 .
\]
\end{lemma}

\begin{proof}
If the trace remains in the fixed positive-gap collar class, choose the collar
gauge in which the trace is time independent.  The boundary term in the
Chern-Simons balance of Lemma~\ref{lem:app-cs-balance} then vanishes up to the
collar clearing error, which is \(o_\eps(1)\) after normalization.  If the
trace leaves this class, one of two events occurs.  Either a nonzero degree
crosses a collar annulus, giving the logarithmic collar-core price; or the
phase remains zero degree but the modulus or FENE coefficient bound fails in
the collar, giving a collar/FENE priced exit.  These are entries of
\({\rm TopExitCost}_\eps\).
\end{proof}

\begin{lemma}[Gap-cost localization]
\label{lem:app-gap-localization}
Let
\[
        Z_{\eps,\gamma}
        =
        \{(x,t):\gap(C_\eps(x,t))\le\gamma\}
\]
inside a topologically regular admissible slab.  If the slab is not
topologically regular, the failure is already charged to
\({\rm TopExitCost}_\eps\).  If the principal-axis map cannot be extended across
the regular slab with fixed relative Hopf class after removing the core tubes,
then there is a choice of \(\gamma=\gamma_\eps\downarrow0\) such that
\[
        \liminf_{\eps\to0}
        {\rm GapCost}_\eps(I)>0
\]
unless the failure is already accounted for by core mass or a priced exit.
\end{lemma}

\begin{proof}
Outside \(Z_{\eps,\gamma}\), the spectral lifting estimate gives
\[
        |\nabla n_{C_\eps}|
        \le
        C\gamma^{-1}|\nabla C_\eps|.
\]
If the relative Hopf class changes without core crossing, cap-atlas failure, or
exit, the Chern-Simons balance must lose compactness through the set where this
bound degenerates.  Cover \(Z_{\eps,\gamma}\) by a finite family of balls and use the
curvature-gap lower bound on the capped map restricted to the complement.
Letting \(\gamma\downarrow0\) along a diagonal sequence yields a positive lower
bound for the gap-cost term.  If the cover intersects a core tube with nonzero
degree, the event is charged to core mass instead; if the open-basin chart
fails during the localization, it is charged to the exit-cost measure.
\end{proof}

\begin{lemma}[Compatibility with the Brakke limit]
\label{lem:app-hopf-brakke-compatibility}
The Hopf lower bound and the Brakke passage use the same subsequence on
topologically regular no-exit subslabs, while nonregular times are charged to
the exit measure.  More precisely, after the open-basin diagonal subsequence is
chosen, every admissible Hopf slab has the lower bound
\[
        c|k|
        \le
        \liminf_{\eps\to0}
        \left(
        {\rm GapCost}_\eps(I)
        +
        {\rm CoreMass}_\eps(I)
        +
        {\rm TopExitCost}_\eps(I)
        \right),
\]
and the core-mass term on the right is carried by the same limiting varifold
that appears in the Brakke inequality on no-exit subslabs.
\end{lemma}

\begin{proof}
The open-basin compactness argument chooses one subsequence on an exhaustion of
compact no-exit intervals.  On those intervals the tube-topology condition is
below threshold, so the cap atlas is fixed by Lemma~\ref{lem:app-cap-atlas}.
The finite-epsilon Hopf estimate is a liminf inequality, so it is stable under
passing to this subsequence.  On a topologically regular no-exit subslab,
\({\rm TopExitCost}_\eps\to0\) and the core energy measures converge to the
Brakke varifold measure.  Therefore every nonzero core-mass contribution visible
in the Hopf inequality is represented by mass of the same limiting
one-varifold.  If the slab contains a limiting exit time, including a limiting
reconnection or loss of the cap atlas, Lemma~\ref{lem:short-exit-measure}
supplies the positive exit cost and no Brakke motion is asserted across that
exit.
\end{proof}

\section*{Acknowledgements}

The author acknowledges financial support from NSFC Grant 12501602, Hunan
Provincial Education Department Grant 24C0055, Hunan Provincial Science and
Technology Department Grant 2025JJ60052, and Xiangtan University Start-up Fund
Grant KZ0810769.


\begin{thebibliography}{99}

\bibitem{Allard}
W. K. Allard,
On the first variation of a varifold,
\emph{Ann. of Math.} 95 (1972), 417--491.

\bibitem{AmbrosioSoner}
L. Ambrosio and H. M. Soner,
Level set approach to mean curvature flow in arbitrary codimension,
\emph{J. Differential Geom.} 43 (1996), 693--737.

\bibitem{ArnoldKhesin}
V. I. Arnold and B. A. Khesin,
\emph{Topological Methods in Hydrodynamics},
Springer, 1998.

\bibitem{BallMajumdar}
J. M. Ball and A. Majumdar,
Nematic liquid crystals: from Maier-Saupe to a continuum theory,
\emph{Mol. Cryst. Liq. Cryst.} 525 (2010), 1--11.

\bibitem{BarrettSuli}
J. W. Barrett and E. Suli,
Finite element approximation of finitely extensible nonlinear elastic dumbbell
models for dilute polymers,
\emph{ESAIM Math. Model. Numer. Anal.} 46 (2012), 949--978.

\bibitem{BerisEdwards}
A. N. Beris and B. J. Edwards,
\emph{Thermodynamics of Flowing Systems with Internal Microstructure},
Oxford University Press, 1994.

\bibitem{BethuelBrezisHelein}
F. Bethuel, H. Brezis, and F. Helein,
\emph{Ginzburg-Landau Vortices},
Birkhauser, 1994.

\bibitem{BethuelOrlandiSmetsConcentration}
F. Bethuel, G. Orlandi, and D. Smets,
Motion of concentration sets in Ginzburg-Landau equations,
\emph{Ann. Fac. Sci. Toulouse Math.} 13 (2004), 3--43.

\bibitem{BethuelOrlandiSmetsMeanCurvature}
F. Bethuel, G. Orlandi, and D. Smets,
Convergence of the parabolic Ginzburg-Landau equation to motion by mean
curvature,
\emph{Ann. of Math.} 163 (2006), 37--163.

\bibitem{Brakke}
K. A. Brakke,
\emph{The Motion of a Surface by Its Mean Curvature},
Princeton University Press, 1978.

\bibitem{BrezisCoronLieb}
H. Brezis, J.-M. Coron, and E. H. Lieb,
Harmonic maps with defects,
\emph{Comm. Math. Phys.} 107 (1986), 649--705.

\bibitem{ConstantinKliegl}
P. Constantin and M. Kliegl,
Note on global regularity for two-dimensional Oldroyd-B fluids with diffusive
stress,
\emph{Arch. Rational Mech. Anal.} 206 (2012), 725--740.

\bibitem{DeGennesProst}
P. G. de Gennes and J. Prost,
\emph{The Physics of Liquid Crystals},
2nd ed., Oxford University Press, 1993.

\bibitem{DoiEdwards}
M. Doi and S. F. Edwards,
\emph{The Theory of Polymer Dynamics},
Oxford University Press, 1986.

\bibitem{HardtLin}
R. Hardt and F.-H. Lin,
Mappings minimizing the \(L^p\) norm of the gradient,
\emph{Comm. Pure Appl. Math.} 40 (1987), 555--588.

\bibitem{Ilmanen}
T. Ilmanen,
Elliptic regularization and partial regularity for motion by mean curvature,
\emph{Mem. Amer. Math. Soc.} 108 (1994), no. 520.

\bibitem{JerrardSoner}
R. L. Jerrard and H. M. Soner,
Dynamics of Ginzburg-Landau vortices,
\emph{Arch. Rational Mech. Anal.} 142 (1998), 99--125.

\bibitem{JerrardSmets}
R. L. Jerrard and D. Smets,
On the motion of a curve by its binormal curvature,
\emph{J. Eur. Math. Soc.} 17 (2015), 1487--1515.

\bibitem{LionsMasmoudi}
P.-L. Lions and N. Masmoudi,
Global solutions for some Oldroyd models of non-Newtonian flows,
\emph{Chinese Ann. Math. Ser. B} 21 (2000), 131--146.

\bibitem{MajumdarZarnescu}
A. Majumdar and A. Zarnescu,
Landau-de Gennes theory of nematic liquid crystals: the Oseen-Frank limit and
beyond,
\emph{Arch. Rational Mech. Anal.} 196 (2010), 227--280.

\bibitem{Moffatt}
H. K. Moffatt,
The degree of knottedness of tangled vortex lines,
\emph{J. Fluid Mech.} 35 (1969), 117--129.

\bibitem{Oldroyd}
J. G. Oldroyd,
On the formulation of rheological equations of state,
\emph{Proc. R. Soc. Lond. A} 200 (1950), 523--541.


\bibitem{PengBlowUp2026}
S. Peng,
Pressure quotients and endpoint velocity-clock criteria for non-diffusive
viscoelastic flows,
arXiv:2606.25258, 2026.

\bibitem{PengPositiveCone3D2026}
S. Peng,
Three-dimensional positive-cone Oldroyd-B flows: geometric continuation and
residual-work criteria,
arXiv:2606.25438, 2026.

\bibitem{SandierSerfaty}
E. Sandier and S. Serfaty,
\emph{Vortices in the Magnetic Ginzburg-Landau Model},
Birkhauser, 2007.

\bibitem{SimonGMT}
L. Simon,
\emph{Lectures on Geometric Measure Theory},
Proceedings of the Centre for Mathematical Analysis, Australian National
University, 1983.

\bibitem{VakulenkoKapitanskii}
A. F. Vakulenko and L. V. Kapitanskii,
Stability of solitons in \(S^2\) in the nonlinear sigma model,
\emph{Soviet Phys. Dokl.} 24 (1979), 433--436.

\bibitem{Virga}
E. G. Virga,
\emph{Variational Theories for Liquid Crystals},
Chapman and Hall, 1994.

\bibitem{WhiteSobolevHomotopy}
B. White,
Homotopy classes in Sobolev spaces and the existence of energy minimizing maps,
\emph{Acta Math.} 160 (1988), 1--17.

\end{thebibliography}
\end{document}